\documentclass[reqno]{amsart}
\usepackage{amssymb,amsmath,amsthm}
\usepackage{color}
\usepackage{enumerate}
\usepackage{fullpage}
\usepackage{url}
\usepackage[Symbol]{upgreek}
\usepackage[unicode,pdfusetitle]{hyperref}
\usepackage[mathscr]{euscript}
\usepackage{tikz-cd}
\usepackage{mathtools}

\usepackage{setspace}
\newtheorem{theorem}{Theorem}
\newtheorem*{theorem*}{Theorem}
\newtheorem{definition}[theorem]{Definition}
\newtheorem*{definition*}{Definition}

\newtheorem*{question*}{Question}

\newtheorem*{conjecture*}{Conjecture}

\newtheorem*{convention*}{Convention}
\newtheorem{assumption}[theorem]{Assumption}
\newtheorem*{assumption*}{Assumption}
\newtheorem*{induction*}{Induction Hypothesis}
\newtheorem{corollary}[theorem]{Corollary}
\newtheorem*{corollary*}{Corollary}
\newtheorem{remark}[theorem]{Remark}
\newtheorem*{remark*}{Remark}
\newtheorem{proposition}[theorem]{Proposition}
\newtheorem*{proposition*}{Proposition}
\newtheorem{lemma}[theorem]{Lemma}
\newtheorem*{lemma*}{Lemma}
\newtheorem{fact}[theorem]{Fact}
\newtheorem*{fact*}{Fact}

\newtheorem*{claim*}{Claim}
\newtheorem{theoremA}{Theorem}

\theoremstyle{definition}
\newtheorem{example}[theorem]{Example}
\newtheorem*{example*}{Example}

\numberwithin{theorem}{section}
\numberwithin{claim}{section}
\numberwithin{equation}{section}

\DeclareMathOperator{\an}{an}

\DeclareMathOperator{\dcl}{dcl}

\DeclareMathOperator{\Graph}{Gr}

\DeclareMathOperator{\RCF}{RCF}

\DeclareMathOperator{\res}{res}

\DeclareMathOperator{\rk}{rk}

\DeclareMathOperator{\Th}{Th}

\newcommand{\N}{\mathbb{N}}

\newcommand{\Q}{\mathbb{Q}}
\newcommand{\R}{\mathbb{R}}
\newcommand{\T}{\mathbb{T}}

\newcommand{\cB}{\mathcal B}
\newcommand{\cC}{\mathcal C}

\newcommand{\cH}{\mathcal H}
\newcommand{\cL}{\mathcal L}

\newcommand{\cO}{\mathcal O}

\newcommand{\cR}{\mathcal R}

\newcommand{\mfM}{\mathfrak{M}}

\newcommand{\fm}{\mathfrak{m}}

\newcommand{\0}{\emptyset}

\newcommand{\dclL}{\dcl_\cL}
\newcommand{\dimL}{\dim_\cL}

\newcommand{\inv}{^{-1}}

\newcommand{\jet}{\mbox{\small$\mathscr{J}$}_\der}
\newcommand{\deltajet}{\mbox{\small$\mathscr{J}$}_\derdelta}
\newcommand{\ojet}[1]{\mbox{\small$\mathscr{J}$}_{#1}}
\newcommand{\Ld}{\cL^\der}
\newcommand{\LdO}{\cL^{\cO,\der}}
\newcommand{\LO}{\cL^{\cO}}

\newcommand{\overbar}[1]{\mkern 2mu\overline{\mkern-2mu#1\mkern-2mu}\mkern 2mu}
\newcommand{\rkL}{\rk_\cL}

\newcommand{\Td}{T^\der}

\newcommand{\TdO}{T^{\cO,\der}}
\newcommand{\TO}{T^{\cO}}

\renewcommand{\preceq}{\preccurlyeq}
\renewcommand{\succeq}{\succcurlyeq}
\renewcommand{\geq}{\geqslant}
\renewcommand{\leq}{\leqslant}
\renewcommand{\epsilon}{\varepsilon}
\renewcommand{\k}{\boldsymbol{k}}

\DeclareFontFamily{OMS}{smallo}{}
\DeclareFontShape{OMS}{smallo}{m}{n}{<->s*[.65]cmsy10}{}
\DeclareSymbolFont{smallo@m}{OMS}{smallo}{m}{n}
\DeclareMathSymbol{\smallo}{\mathord}{smallo@m}{79}

\DeclareFontFamily{U}{fsy}{}
\DeclareFontShape{U}{fsy}{m}{n}{<->s*[.9]psyr}{}
\DeclareSymbolFont{der@m}{U}{fsy}{m}{n}
\DeclareMathSymbol{\der}{\mathord}{der@m}{182}
\DeclareSymbolFont{der@m}{U}{fsy}{m}{n}
\DeclareMathSymbol{\derdelta}{\mathord}{der@m}{100}
\author{Elliot Kaplan}
\email{kaplae2@mcmaster.ca}
\title{$T$-convex $T$-differential fields and their immediate extensions}
\address{Department of Mathematics and Statistics, McMaster University, 1280 Main Street West, Hamilton, Ontario, L8S 4K1, Canada}
\begin{document}
\maketitle

\begin{abstract}
Let $T$ be a polynomially bounded o-minimal theory extending the theory of real closed ordered fields. Let $K$ be a model of $T$ equipped with a $T$-convex valuation ring and a $T$-derivation. If this derivation is continuous with respect to the valuation topology, then we call $K$ a \emph{$T$-convex $T$-differential field}. We show that every $T$-convex $T$-differential field has an immediate strict $T$-convex $T$-differential field extension which is spherically complete. In some important cases, the assumption of polynomial boundedness can be relaxed to power boundedness.
\end{abstract}


\section*{Introduction}
\noindent
In this article, $T$ is a complete, model complete o-minimal theory which extends the theory $\RCF$ of real closed ordered fields in some appropriate language $\cL \supseteq \{0,1,+,-, \cdot,<\}$. Let $K\models T$ and let $\cO$ be a $T$-convex valuation ring of $K$ (that is, a convex subset of $K$ which is closed under all unary $\cL(\emptyset)$-definable continuous functions on $K$; introduced in~\cite{DL95}). Let also $\der$ be a $T$-derivation on $K$ as defined in~\cite{FK19}, that is, $\der$ is a map $K\to K$ which satisfies the identity $\der F(u) = \nabla F(u)\cdot \der u$ for all $u \in K^n$ and all $\cL(\emptyset)$-definable functions $F$ which are $\cC^1$ at $u$. If the $T$-derivation $\der$ is continuous with respect to the valuation topology induced by $\cO$, then we call $K = (K,\cO,\der)$ a \textbf{$T$-convex $T$-differential field}. Motivating examples include $\cR$-Hardy fields, defined in~\cite{DMM94}, and the expansion of the field of logarithmic-exponential transseries by restricted analytic functions, considered in~\cite{DMM97}; see Examples~\ref{ex-hardy} and~\ref{ex-transseriesTdO} below.

\medskip\noindent
Let $K$ be a $T$-convex $T$-differential field. By~\cite[Lemma 4.4.7]{ADH17}, continuity of the derivation is equivalent to the existence of $\phi \in K^\times$ with $\der\smallo \subseteq \phi \smallo$, where $\smallo$ is the unique maximal ideal of $\cO$. In the special case that $\der\smallo \subseteq \smallo$, we say that $K$ has \textbf{small derivation}. Following~\cite{ADH18B}, we say that a $T$-convex $T$-differential field extension $M$ of $K$ is \textbf{strict} if
\[
\der \smallo \subseteq \phi\smallo\ \Longrightarrow\ \der_M\smallo_M \subseteq \phi\smallo_M,\qquad\der \cO \subseteq \phi\smallo\ \Longrightarrow\ \der_M\cO_M \subseteq \phi\smallo_M
\]
for each $\phi \in K^\times$. In this article, we show the following:

\begin{theoremA}[Corollary~\ref{cor-polycase}]
\label{thm-polystrict}
If $T$ is polynomially bounded, then $K$ has an immediate strict $T$-convex $T$-differential field extension which is spherically complete.
\end{theoremA}

\noindent
When $T = \RCF$, Theorem~\ref{thm-polystrict} says that every real closed valued field equipped with a continuous derivation has a spherically complete immediate strict extension, a result previously established in~\cite{ADH18B}. A case unique to this article is when $T = T_{\an}\coloneqq \Th(\R_{\an})$, the theory of the expansion of the real field by all restricted analytic functions. The theory $T_{\an}$ is model complete, o-minimal, and polynomially bounded~\cite{vdD86,Ga68}. We give an example of a spherically complete $T_{\an}$-convex $T_{\an}$-differential field in Example~\ref{ex-hahnTdO} below. Valued fields equipped with both analytic structure and an operator have been studied before by Rideau~\cite{Ri17}. Our setting, however, is quite different from that of Rideau. 

\medskip\noindent
In some important cases, the assumption of polynomial boundedness can be relaxed to power boundedness (a generalization of polynomial boundedness introduced in~\cite{Mi96}).

\begin{theoremA}[Corollaries~\ref{cor-nontrivres} and~\ref{cor-asymp}]
\label{thm-resasymp}
Suppose $T$ is power bounded.
\begin{enumerate}
\item[(i)] If $K$ has small derivation and the induced derivation on the residue field of $K$ is nontrivial, then $K$ has a spherically complete immediate $T$-convex $T$-differential field extension with small derivation. 
\item[(ii)] If $K$ is asymptotic (that is, for all nonzero $f,g\in \smallo$, we have $f / g\in \smallo \Longleftrightarrow \der f/\der g\in \smallo$), then $K$ has a spherically complete immediate asymptotic $T$-convex $T$-differential field extension.
\end{enumerate}
\end{theoremA}

\noindent
In both cases (i) and (ii) above, the spherically complete immediate extensions of $K$ are necessarily strict; this is readily verified in case (i) and it follows from~\cite[Lemma 1.11]{ADH18B} in case (ii). Theorems~\ref{thm-polystrict} and~\ref{thm-resasymp} both follow from something a bit more general. Before stating this, we recall some definitions from~\cite{ADH18B}. Let $v\colon K^\times \to \Gamma$ be the (surjective) Krull valuation corresponding to $\cO$ and set
\[
\Gamma(\der)\ \coloneqq\ \big\{v(\phi) :\phi \in K^\times \text{ and } \der\smallo\subseteq \phi\smallo\big\},\qquad S(\der)\ \coloneqq\ \big\{\gamma\in \Gamma:\Gamma(\der)+\gamma = \Gamma(\der)\}.
\]

\begin{theoremA}[Theorem~\ref{thm-main2}]
\label{thm-mainthm}
Suppose that $T$ is power bounded with field of exponents $\Lambda$ and that $S(\der)$ is a $\Lambda$-subspace of $\Gamma$. Then $K$ has an immediate strict $T$-convex $T$-differential field extension which is spherically complete.
\end{theoremA}

\noindent
The condition that $S(\der)$ is a $\Lambda$-subspace of $\Gamma$ is satisfied when $\Lambda$ is archimedean (that is, when $T$ is polynomially bounded). It is also satisfied when $S(\der) = \{0\}$, a situation which arises when $K$ is asymptotic or when $K$ has small derivation and the induced derivation on the residue field of $K$ is nontrivial. It is unclear whether $S(\der)$ is always a $\Lambda$-subspace of $\Gamma$, or whether this assumption is necessary.

\medskip\noindent
The assumption of power boundedness, on the other hand, we know to be necessary. If $T$ is not power bounded and $\cO \neq K$, then $K$ has no $T$-convex extension which is spherically complete by Miller's dichotomy~\cite{Mi96} and a negative result of Kuhlmann, Kuhlmann, and Shelah~\cite{KKS97}, established independently by van der Hoeven~\cite[Proposition 2.2]{vdH97}. See Remark~\ref{rem-noimmediateext} below for details.

\medskip\noindent
Of course, if $\cO = K$, then $K$ is itself spherically complete and Theorem~\ref{thm-main2} holds vacuously (even without the assumption of power boundedness). Thus, we assume throughout the article that $\cO$ is a \emph{proper} $T$-convex subring of $K$.
\subsection*{Outline}
In~\cite{ADH18B}, Aschenbrenner, van den Dries, and van der Hoeven proved that every equicharacteristic zero valued field with a continuous derivation has an immediate strict valued differential field extension which is spherically complete. Our proof of Theorem~\ref{thm-mainthm} is similar in structure to the proof of this result. As in~\cite{ADH18B}, we first handle the case $S(\der) = \{0\}$, and then we reduce to this case using a coarsening argument. However, since we work with $\cL(K)$-definable functions instead of differential polynomials over $K$, many of the key tools from~\cite{ADH18B} are not available to us. Below is an outline of our strategy.

\medskip\noindent
Suppose $T$ is power bounded and let $K$ be a $T$-convex $T$-differential field. Assume that $K$ is not spherically complete, so there is a collection $\cB$ of sets of the form $\big\{y \in K: v(y-a)\geq\gamma\big\}$ such that any two sets in $\cB$ intersect, but the intersection $\bigcap\cB$ is empty. Compactness allows us to find an element $\ell$ in a strict $T$-convex $T$-differential field extension of $K$ such that $\ell$ lies in (the natural extension of) every set in $\cB$. There are two possibilities:
\begin{enumerate} 
\item[(1)] There is some element $a$ realizing the same cut as $\ell$ in some strict $T$-convex $T$-differential field extension of $K$ and some $r$ such that $a^{(r)}$ is in the $\cL$-definable closure of $K\cup \{a,a',\ldots,a^{(r-1)}\}$, where $a'\coloneqq\der a$ and $a^{(r)}\coloneqq\der^r a$.
\item[(2)] There is no such element in any strict $T$-convex $T$-differential field extension of $K$.
\end{enumerate}
In the first case, we are tasked with showing that $K\langle a,a',\ldots\rangle$ is an immediate extension of $K$, where $K\langle a,a',\ldots\rangle$ is the $\cL$-definable closure of $K \cup \{a,a',\ldots\}$. In the second case, we need to show that $K\langle \ell,\ell',\ldots\rangle$ is an immediate extension of $K$.

\medskip\noindent
First, consider the simplest situation: assume that $K$ has small derivation and that the induced derivation on the residue field of $K$ is nontrivial. Suppose we are in case (1), and take an $\cL(K)$-definable function $I_F$ such that $a^{(r)}=I_F(a,\ldots,a^{(r-1)})$. Set
\[
F(y,\ldots,y^{(r)})\ \coloneqq\ y^{(r)}-I_F(y,\ldots,y^{(r-1)}),
\]
so $F(a,\ldots,a^{(r)}) = 0$. We look at how $F$ behaves on arbitrarily small neighborhoods of $\ell$ via \emph{affine conjugation}: we replace $\ell$ with $d\inv(\ell-b)$ for $b\in K$ and $d \in K^\times$ with $v(\ell-b)>vd$, and we adjust $F$ accordingly. After affine conjugation, we get that $\ell \in \smallo$, so $y \in \smallo$ for $y \in K$ sufficiently close to $\ell$. Since any such $y \in K$ is also close to $a$, we should expect that $F(y,\ldots,y^{(r)}) \in \smallo$ for these $y$. If this happens for all sufficiently small $d$, then we say that $F$ \emph{vanishes at $(K,\ell)$}. 

\medskip\noindent
Vanishing turns out to be the dividing line between cases (1) and (2). If there is no function $F$ which vanishes at $(K,\ell)$, then we must be in case (2), and a somewhat technical inductive argument (Proposition~\ref{prop-imsim}) can be used to show that $vG(y,\ldots,y^{(n)})= vG(\ell,\ldots,\ell^{(n)})$ for all $\cL(K)$-definable functions $G$ and all $y \in K$ sufficiently close to $\ell$. From this, we see that $K\langle \ell,\ell',\ldots\rangle$ is an immediate strict extension of $K$; see Proposition~\ref{prop-zerosetisempty}. On the other hand, suppose that $F$ is a function of the above form which vanishes at $(K,\ell)$. The same argument as in the nonvanishing case allows us to find an immediate $T$-convex $T$-differential field extension $K\langle a,a',\ldots\rangle$ of $K$ with $F(a,\ldots,a^{(r)}) = 0$. If, towards contradiction, this extension is not strict, then one can show that $vF(y,\ldots,y^{(r)})=vF(z,\ldots,z^{(r)})$ for all $y,z \in K$ sufficiently close to $\ell$ (this proof by contradiction is carried out in the proof of Proposition~\ref{prop-zerosetisnonempty}). Thus, we need to prove that for any $y \in K$, there is $z \in K$ with $v(\ell-z)>v(\ell- y)$ such that $vF(y,\ldots,y^{(r)})\neq vF(z,\ldots,z^{(r)})$. After a suitable affine transformation and rescaling, this reduces to showing that there are $y,z \in \cO$ with $F(y,\ldots,y^{(r)})\in \smallo$ and $F(z,\ldots,z^{(r)})\not\in \smallo$. We let
\[
S\ \coloneqq \ \big\{y \in \cO: F(y,\ldots,y^{(r)})\in \smallo\big\}.
\]
Our assumption that $F$ vanishes at $(K,\ell)$ guarantees that $S$ is nonempty. To see that $\cO \setminus S$ is nonempty, let $\bar{S}$ be the image of $S$ in the differential residue field $\res(K)$, and note that $\bar{S}$ is \emph{thin} (it is controlled by an $\cL$-definable subset of the residue field with empty interior; see Subsection~\ref{subsec-thin} for a precise definition). In Proposition~\ref{prop-nondefinable} below, we show that any model of $T$ equipped with a nontrivial $T$-derivation is not thin, so our assumption that $\res(K)$ has nontrivial derivation allows us to conclude that $\bar{S} \neq \res(K)$. This gives that $\cO \setminus S$ is nonempty as desired; see Lemma~\ref{lem-nontrivialbound} for the full argument.

\medskip\noindent
We now drop the assumption that the derivation on the residue field is nontrivial, but we still assume that $S(\der) = \{0\}$. This is the setting for most of this article, and Theorem~\ref{thm-mainthm} is first proven under this assumption in Theorem~\ref{thm-main1} below. In this setting, we can approximate the previous case by \emph{compositionally conjugating}, that is, by replacing $\der$ with $\phi\inv\der$ for $\phi \in K^\times$ with $v(\phi)\in \Gamma(\der)$; see Subsection~\ref{subsec-compconj}. As $v(\phi)$ grows larger, the compositional conjugate $K^\phi$ behaves more and more like a $T$-convex $T$-differential field with small derivation and nontrivial differential residue field. This method of studying \emph{eventual behavior} via compositional conjugation was developed in~\cite{ADH17} and generalized in~\cite{ADH18B} to study immediate strict extensions of valued differential fields. Our definition of vanishing needs to be altered slightly; we now require that $F$ vanish at $(K,\ell)$ after compositionally conjugating by $\phi \in K^\times$ with $v(\phi)\in \Gamma(\der)$ sufficiently large.

\medskip\noindent
Even with this more general definition of vanishing, we can still show that if no function vanishes at $(K,\ell)$, then $vG(y,\ldots,y^{(n)})=vG(\ell,\ldots,\ell^{(n)})$ for all $G$ and all $y \in K$ sufficiently close to $\ell$, so $K\langle \ell,\ell',\ldots\rangle$ is an immediate strict extension of $K$. We can also show that if $F$ vanishes at $(K,\ell)$, then $K$ has an immediate $T$-convex $T$-differential field extension $K\langle a,a',\ldots\rangle$ with $F(a,\ldots,a^{(r)}) = 0$. As above, proving that this extension is strict reduces to showing that for any $y \in K$, there is $z \in K$ with $v(\ell-z)>v(\ell- y)$ such that $vF(y,\ldots,y^{(r)})\neq vF(z,\ldots,z^{(r)})$. Though we can no longer apply our thin set argument directly, we can use our assumption that $S(\der) = \{0\}$ and a compactness argument to find an elementary extension $M$ of $K$ which, after compositional conjugation and coarsening, has nontrivial differential residue field. We can then apply the thin set argument (Lemma~\ref{lem-nontrivialbound}) to $M$ to find $u \in K$ such that $vF(u,\ldots,u^{(r)})$ is sufficiently small. This is all done in the proof of Lemma~\ref{lem-approxbound}. Going from the conclusion of Lemma~\ref{lem-approxbound} to the existence of $y$ and $z$ as above involves a few technical steps, which are carried out in Corollary~\ref{cor-sameval} and Lemmas~\ref{lem-differentvals} and~\ref{lem-differentclosevals}. 

\medskip\noindent
In the final section, we further relax our assumptions to match the statement of Theorem~\ref{thm-mainthm}. That is, we assume only that $\Delta \coloneqq S(\der)$ is a $\Lambda$-subspace of $\Gamma$. Following~\cite[Section 6]{ADH18B}, we arrange that $K$ has small derivation using compositional conjugation, and we decompose $K$ into two structures. The first is the \emph{$\Delta$-coarsening of $K$}, denoted $K_\Delta$, which has the same underlying $T$-differential field as $K$ but now has $T$-convex valuation ring
\[
\cO_{K_\Delta}\ \coloneqq \ \big\{y \in K: v(y)> \gamma\text{ for some }\gamma \in \Delta\big\}.
\]
The second is the \emph{$\Delta$-specialization of $K$}, which is the residue field $\res(K_\Delta)$ of the $\Delta$-coarsening of $K$, considered as a $T$-convex $T$-differential field in its own right. As it turns out, $S_{K_\Delta}(\der) = \{0\}$ and the derivation on $\res(K_\Delta)$ is trivial. Our result in the case when $S(\der) = \{0\}$ (Theorem~\ref{thm-main1}) gives us a spherically complete immediate strict extension of $K_\Delta$. Corollary~\ref{cor-immediateext} provides a spherically complete immediate extension of $\res(K_\Delta)$, and we equip this extension with the trivial derivation, which makes it a strict extension of $\res(K_\Delta)$. Using a result on residue field extensions (Corollary~\ref{cor-resext}), we can glue these two extensions together to find a spherically complete immediate strict extension of $K$. This gluing is done in the proof of Theorem~\ref{thm-main2}.

\subsection*{Acknowledgements}
Research for this article was conducted at the University of Illinois Urbana-Champaign, and the results first appeared in my PhD thesis~\cite{Ka21}. I would like to thank the referees for useful comments, Allen Gehret and Erik Walsberg for helpful conversations, and especially Nigel Pynn-Coates and Lou van den Dries for providing feedback on earlier drafts of this article. 

\subsection*{Notation and conventions}\label{subsec-notation}
In this article, we always use $k$, $m$, $n$, $p$, $q$, and $r$ to denote elements of $\N =\{0,1,2,\ldots\}$. By ``ordered set'' we mean ``totally ordered set.'' Let $S$ be an ordered set, let $a \in S$, and let $A \subseteq S$. We set $S^{>a} \coloneqq \{s\in S:s>a\}$; similarly for $S^{\geq a}$, $S^{<a}$, $S^{\leq a}$, and $S^{\neq a}$. We write ``$a>A$'' (respectively ``$a<A$'') if $a$ is greater (less) than each $s \in A$. For $b \in S^{>a}$, we put
\[
[a,b]_A\ \coloneqq \ \{s \in A:a\leq s\leq b\}.
\]
If $A = S$, then we drop the subscript and write $[a,b]$ instead. A \textbf{cut} in $S$ is just a downward closed subset of $S$. If $A$ is a cut in $S$ and $y$ is an element in an ordered set extending $S$, then we say that \textbf{$y$ realizes the cut $A$} if $A < y < S\setminus A$. If $\Gamma$ is an ordered abelian group, then we set $\Gamma^> \coloneqq \Gamma^{>0}$, and we define $\Gamma^{\geq}$, $\Gamma^<$, $\Gamma^{\leq}$, and $\Gamma^{\neq}$ analogously. If $R$ is a ring, then we let $R^\times$ denote the multiplicative group of units in $R$.

\medskip\noindent
We always use $K$, $L$, and $M$ for models of $T$ (or expansions thereof). We regard $K^0$ as the one-point space $\{0\}$, and we identify each nullary map $F\colon K^0\to K^n$ with its value $F(0) \in K^n$. Let $A \subseteq K$ and let $D\subseteq K^n$. We say that $D$ is \textbf{$\cL(A)$-definable} if
\[
D\ =\ \varphi(K) \ \coloneqq \ \big\{y \in K^n:K \models \varphi(y)\big\}
\]
for some $\cL(A)$-formula $\varphi(y)$. If $D\subseteq K^n$ is $\cL(A)$-definable, then we let $\dimL(D)$ denote the o-minimal dimension of $D$. Let $k \leq n$. We denote the projection of $D$ onto the first $k$ coordinates by $\pi_k(D)$ and for $y \in K^k$, we set $D_y \coloneqq \big\{z \in K^{n-k}:(y,z) \in D\big\}$.

\medskip\noindent
Let $U \subseteq K^n$ and let $F\colon U \to K$ be a function. We let $\Graph(F)\subseteq K^{n+1}$ denote the graph of $F$, and we say that $F$ is $\cL(A)$-definable if $\Graph(F)$ is. Note that the domain of an $\cL(A)$-definable function is $\cL(A)$-definable. If $F$ is $\cC^1$ and $U$ is open, then we let $\nabla F$ denote the gradient
\[
\nabla F\ \coloneqq \ \left(\frac{\partial F}{\partial Y_1},\ldots,\frac{\partial F}{\partial Y_n}\right),
\]
viewed as an $\cL(K)$-definable map from $U$ to $K^n$. If $n = 1$, then we write $F'$ instead of $\nabla F$.

\medskip\noindent
For $A \subseteq K$, we let $\dclL(A)$ be the $\cL$-definable closure of $A$ (in $K$, implicitly, but this doesn't change if we pass to elementary extensions of $K$). If $b \in \dclL(A)$, then $b = F(a)$ for some $\cL(\emptyset)$-definable function $F$ and some tuple $a$ from $A$. It is well-known that $(K,\dclL)$ is a pregeometry. A set $B \subseteq K$ is said to be \textbf{$\cL(A)$-independent} if $b\not\in \dclL\big(A\cup(B \setminus \{b\})\big)$ for all $b \in B$. A tuple $(a_i)_{i \in I}$ is said to be $\cL(A)$-independent if its set of components $\{a_i:i \in I\}$ is $\cL(A)$-independent and no components are repeated. 

\medskip\noindent
Let $M$ be a $T$-extension of $K$, that is, a model of $T$ which contains $K$ as an $\cL$-substructure. Given an $\cL(K)$-definable set $D \subseteq K^n$, we let $D^M$ denote the subset of $M^n$ defined by the same $\cL(K)$-formula as $D$. We sometimes refer to $D^M$ as the \textbf{natural extension of $D$ to $M$}. Since $T$ is assumed to be model complete, this natural extension does not depend on the choice of defining formula. If $F\colon U \to K^m$ is an $\cL(K)$-definable function, then we let $F^M\colon U^M\to M$ be the $\cL(K)$-definable map with graph $\Graph(F^M) = \Graph(F)^M$. We often drop the superscript for definable maps and just write $F\colon U^M\to M$. 

\medskip\noindent
Let $A \subseteq M$. We let $K\langle A\rangle$ denote the $\cL$-substructure of $M$ with underlying set $\dclL(K\cup A)$. If $A = \{a_1,\ldots,a_n\}$, we write $K\langle a_1,\ldots,a_n\rangle$ instead of $K\langle A \rangle$. Since $T$ has definable Skolem functions, $K\langle A \rangle$ is an elementary $\cL$-substructure of $M$. If $A$ is $\cL(K)$-independent and $M = K\langle A \rangle$, then $A$ is called a \textbf{basis for $M$ over $K$}. The \textbf{rank of $M$ over $K$}, denoted $\rkL(M|K)$, is the cardinality of a basis for $M$ over $K$ (this doesn't depend on the choice of basis). We say that $M$ is a \textbf{simple extension} of $K$ if $\rkL(M|K) = 1$. Then $M = K\langle a\rangle$ for some $a \in M\setminus K$.

\medskip\noindent
Let $\cL^*\supseteq \cL$, let $T^*$ be an $\cL^*$-theory extending $T$, and let $K \models T^*$. We use the same conventions for $\cL^*$-definability as we do for $\cL$-definability. A \textbf{$T^*$-extension of $K$} is a model $M \models T^*$ which contains $K$ as an $\cL^*$-substructure. If $M$ is an \emph{elementary} $T^*$-extension of $K$ and $D \subseteq K^n$ is $\cL^*(K)$-definable, then we let $D^M$ denote the subset of $M^n$ defined by the same formula as $D$.

\section{$T$-convex valuation rings}\label{sec-Tconvex}
\noindent
The fundamentals of valuation theory on o-minimal fields were established by van den Dries and Lewenberg in~\cite{DL95}. In this section, we set up valuation theoretic notation, recall some important results, and establish lemmas for later use. 

\medskip\noindent
Following~\cite{DL95}, we say that a nonempty convex set $\cO \subseteq K$ is a \textbf{$T$-convex valuation ring of $K$} if $F(\cO) \subseteq \cO$ for all $\cL(\emptyset)$-definable continuous functions $F\colon K \to K$. Let $\LO \coloneqq \cL\cup \{\cO\}$ be the extension of $\cL$ by a unary predicate $\cO$ and let $\TO$ be the $\LO$-theory which extends $T$ by axioms asserting that $\cO$ is a \emph{proper} $T$-convex valuation ring. 

\begin{fact}[\cite{DL95}, Corollary 3.13]
\label{fact-completeTconvex}
The theory $\TO$ is complete and model complete.
\end{fact}

\noindent
For the rest of this section, let $K = (K,\cO) \models \TO$ (so $K$ is now an $\LO$-structure). The following is an easy consequence of the o-minimal monotonicity theorem:

\begin{fact}
\label{fact-hull}
The convex hull of an elementary $\cL$-substructure of $K$ is a $T$-convex valuation ring of $K$. 
\end{fact}

\noindent
Our primary reference for valuation theory is~\cite[Chapter 3]{ADH17}, and we make use of the notation used there. Let $\smallo$ denote the unique maximal ideal of $\cO$, and let $\Gamma$ denote the \textbf{value group} of $(K,\cO)$, with (surjective) Krull valuation $v\colon K^\times\to \Gamma$. We set $\Gamma_\infty \coloneqq \Gamma\cup\{\infty\}$, and we extend $v$ to all of $K$ by setting $v(0) \coloneqq \infty$. For $a,b \in K$, we set $a \preceq b :\Longleftrightarrow va\geq vb$. We define $a\prec b$ and $a \asymp b$ similarly, and we set $a\sim b:\Longleftrightarrow a-b \prec a$. Then $\sim$ is an equivalence relation on $K^\times$ and if $a\sim b$, then $a\asymp b$ and $a$ is positive if and only if $b$ is.

\medskip\noindent
We denote the \textbf{residue field of $K$} by $\res(K)$, and we let $a\mapsto \bar{a}$ denote the residue map $\cO \to \res(K)$. Under this map, $\res(K)$ admits a natural expansion to a $T$-model; see~\cite[Remark 2.16]{DL95} for details. A \textbf{lift of $\res(K)$} is an elementary $\cL$-substructure $\k$ of $K$ contained in $\cO$ such that the map $a\mapsto\bar{a}\colon\k\to \res(K)$ is an $\cL$-isomorphism. By~\cite[Theorem 2.12]{DL95}, we can always find a lift of $\res(K)$. For $a = (a_1,\ldots,a_n) \in \cO^n$, we let $\bar{a} \coloneqq (\bar{a}_1,\ldots,\bar{a}_n) \in \res(K)^n$, and for $D \subseteq K^n$, we let
\[
\overbar{D}\ \coloneqq \ \{\bar{a}:a \in D\cap \cO^n\}\ \subseteq\ \res(K)^n.
\]

\begin{fact}[\cite{vdD97}, Proposition 1.10]
\label{fact-resdef}
If $D\subseteq K^n$ is $\cL(K)$-definable, then $\overbar{D}$ is $\cL(\res K)$-definable and $\dimL\overbar{D} \leq \dimL D$.
\end{fact}

\noindent
Let $M$ be a $\TO$-extension of $K$ with $T$-convex valuation ring $\cO_M$. We view $\Gamma$ as a subgroup of $\Gamma_M$ and $\res (K)$ as an $\cL$-substructure of $\res(M)$ in the obvious way. We let $v$ and $x \mapsto \bar{x}$ denote their extensions to functions $M^\times \to \Gamma_M$ and $\cO_M\to \res(M)$. 

\medskip\noindent
Let $\ell$ be an element in a $\TO$-extension of $K$ and suppose that the set
\[
v(\ell-K)\ \coloneqq \ \big\{v(\ell-y):y \in K\big\}
\]
is contained in $\Gamma$ and has no largest element. Let $M$ be a $\TO$-extension of $K$ and let $\eta \in v(\ell- K)$. An element $y \in M$ is said to be \textbf{$\eta$-close to $\ell$} if there is $a \in K$ with $v(\ell-a), v(y-a) > \eta$. Note that $M$ is not assumed to contain $\ell$. A property is said to hold \textbf{for all $y \in M$ sufficiently close to $\ell$} if there exists $\eta\in v(\ell-K)$ such that for all $y \in M$ which are $\eta$-close to $\ell$, the property holds. This terminology will most frequently be used in the case $M = K$, but we allow $M \neq K$ for additional flexibility.

\medskip\noindent
The following Lemma on simple residue field extensions will be used in Section~\ref{sec-coarsening}:

\begin{lemma}\label{lem-smallderivres}
Let $M=K\langle a \rangle$ be a simple $\TO$-extension of $K$ with $\Gamma_M = \Gamma$, $a\asymp 1$, and $\bar{a} \not\in \res(K)$. Let $F\colon K\to K$ be an $\cL(K)$-definable function. Then $F'(a) \preceq F(a)$.
\end{lemma}
\begin{proof}
Let $\k\subseteq \cO^\times$ be a lift of $\res(K)$, so $\k\langle a \rangle$ is a lift of $\res(M)$ by~\cite[Lemma 5.1]{DL95}. Using that $\Gamma_M = \Gamma$, take $b \in K^>$ with $F(a)\asymp b$. We need to show that $F'(a) \preceq b$. Since $\Gamma^<$ has no largest element, it suffices to show that $b\inv|F'(a)| < d$ for each $d \in K^>$ with $d \succ 1$. Let such an element $d$ be given. By $\cL$-elementarity, it is enough to show that for any subinterval $I\subseteq K^>$ with $a \in I^M$, there is $y \in I$ with $b\inv|F'(y)|< d$. Let $I$ be such an interval and take an $\cL(\k)$-definable function $G\colon K \to K$ with $b\inv|F(a)|<G(a)$. By shrinking $I$, we arrange that $F$ is $\cC^1$ on $I$ and that $b\inv|F(y)|<G(y)$ for all $y \in I$. As $\bar{a} \in \overbar{I}^{\res(M)}$, we see that $\overbar{I}$ must be infinite, so $I \cap \k$ is infinite. Take $y_1,y_2 \in I\cap \k$ with $y_1<y_2$, so $y_2-y_1 \asymp 1$. Note that $G(y_i)\in \k$, so $b\inv|F(y_i)| < G(y_i) \prec d$ for $i = 1,2$. By the o-minimal mean value theorem, we have
\[
b\inv F'(y)\ =\ \frac{b\inv F(y_2)-b\inv F(y_1)}{y_2-y_1}\ \prec\ d
\]
for some $y \in I$ between $ y_1$ and $ y_2$. In particular, $b\inv|F'(y)|<d$. 
\end{proof}

\subsection{Immediate extensions}
\noindent
In this subsection, let $M$ be a $\TO$-extension of $K$. If $\Gamma_M = \Gamma$ and $\res(M) = \res(K)$, then $M$ is said to be an \textbf{immediate extension of $K$}. Note that $M$ is an immediate extension of $K$ if and only if for all $a \in M^\times$ there is $b \in K^\times$ with $a \sim b$. The next lemma shows that in an immediate extension of $K$, we can approximate $\LO(K)$-definable sets by $\cL(K)$-definable sets.

\begin{lemma}
\label{lem-Lapprox}
Suppose $M$ is an immediate extension of $K$, let $A \subseteq K^n$ be an $\LO(K)$-definable set, and let $a \in A^M$. Then there is an $\cL(K)$-definable cell $D \subseteq A$ with $a \in D^M$.
\end{lemma}
\begin{proof}
It suffices to find an $\cL(K)$-definable set $B \subseteq A$ with $a \in B^M$, for then we can replace $B$ with a subcell $D$ in a cell decomposition of $B$. Since $T$ has definable Skolem functions, we may arrange that $T$ has quantifier elimination and a universal axiomatization by extending $\cL$ by function symbols for all $\cL(\emptyset)$-definable functions. Then $\TO$ eliminates quantifiers by~\cite{DL95}, so
\[
A\ =\ \bigcup_{i \leq m}\bigcap_{j \leq k} A_{i,j}
\]
where either $A_{i,j}$ is $\cL(K)$-definable or 
\[
A_{i,j}\ =\ \big\{y \in K^n:F(y) \in \cO\big\}\ \text{ or }\ A_{i,j}\ =\ \big\{y \in K^n:F(y) \not\in \cO\big\}
\]
for some $\cL(K)$-definable function $F\colon K^n\to K$. For each $i\leq m$ and each $j \leq k$, we take an $\cL(K)$-definable set $B_{i,j}\subseteq A_{i,j}$ such that if $a \in A_{i,j}^M$, then $a \in B_{i,j}^M$. We do this as follows.
\begin{enumerate}
\item[(i)] If $A_{i,j}$ is $\cL(K)$-definable, then we set $B_{i,j} \coloneqq A_{i,j}$.
\item[(ii)] Suppose $A_{i,j} = \big\{y \in K^n:F(y) \in \cO\big\}$ for some $\cL(K)$-definable $F$. If $F(a) \not\in \cO_M$, then we set $B_{i,j} \coloneqq \emptyset$. If $F(a) \in \cO_M$, then since $\res(M)=\res(K)$, we may take $u \in K^>$ with $u \asymp 1$ and $|F(a)|<u$. We set 
\[
B_{i,j}\ \coloneqq \ \big\{y \in K^n:|F(y)|<u\big\}.
\]
\item[(iii)] Suppose $A_{i,j} = \big\{y \in K^n:F(y) \not\in \cO\big\}$ for some $\cL(K)$-definable $F$. If $F(a) \in \cO_M$, then we set $B_{i,j} \coloneqq \emptyset$. If $F(a) \not\in \cO_M$, then since $\Gamma_M=\Gamma$, we may take $d \in K^>$ with $d\succ 1$ and $|F(a)|>d$. We set 
\[
B_{i,j}\ \coloneqq \ \big\{y \in K^n:|F(y)|>d\big\}.
\]
\end{enumerate}
Now set 
\[
B\ \coloneqq \ \bigcup_{i \leq m}\bigcap_{j \leq k} B_{i,j}.\qedhere
\]
\end{proof}

\begin{corollary}
\label{cor-Lflat}
Suppose $M$ is an immediate extension of $K$, let $F\colon A\to K$ be an $\LO(K)$-definable function, and let $a \in A^M$. Then there is an $\cL(K)$-definable cell $D \subseteq A$ with $a \in D^M$ such that either $F(y) = 0$ for all $y\in D^M$ or $F(y) \sim F(a)$ for all $y \in D^M$.
\end{corollary}
\begin{proof}
If $F(a) = 0$, then apply Lemma~\ref{lem-Lapprox} to the $\LO(K)$-definable set $\big\{ y\in A:F(y) =0\big\}$.
If $F(a)\neq 0$, then take $b \in K^\times$ with $F(a)\sim b$ and apply Lemma~\ref{lem-Lapprox} to the $\LO(K)$-definable set $\big\{ y\in A:F(y) \sim b\big\}$.
\end{proof}

\noindent
Let $\ell$ be an element in an immediate $\TO$-extension of $K$. Then the set $v(\ell-K)$ is contained in $\Gamma$ and has no largest element. To see this, let $y \in K$ be given and take $b \in K$ with $\ell-y\sim b$. Then $v(\ell-y-b)>v(\ell-y)= v(b) \in \Gamma$. These values $v(\ell-y)$ completely determine the extension $K\langle \ell\rangle$ up to $\LO(K)$-isomorphism\footnote{Model-theoretically speaking, this means that the $\LO$-type of $\ell$ over $K$ is determined by these values.}:

\begin{corollary}
\label{cor-immediateext1}
Let $K\langle \ell \rangle$ be a simple immediate $\TO$-extension of $K$ and let $a\in M$ with $v(a-y) = v(\ell-y) \in \Gamma$ for each $y \in K$. Then there is a unique $\LO(K)$-embedding $K\langle \ell \rangle\to M$ sending $\ell$ to $a$.
\end{corollary}
\begin{proof}
First, we will show that $a$ and $\ell$ realize the same cut in $K$. Let $y \in K$ with $y<\ell$ and take $f \in K^>$ with $\ell-y \sim f$. Then $\ell-y-f \prec f$, so $a-y-f \asymp \ell-y-f \prec f$. Thus $a-y \sim f > 0$. Likewise, if $y\in K$ and $y>\ell$, then $y>a$. This gives us a unique $\cL(K)$-embedding $\iota\colon K\langle \ell \rangle\to M$ sending $\ell$ to $a$. To get that $\iota$ is an $\LO(K)$-embedding, let $F\colon K\to K$ be $\cL(K)$-definable. We need to show that $F(\ell) \in \cO_{K\langle \ell \rangle}$ if and only if $F(a) \in \cO_M$. We assume that $F(\ell) \neq 0$, and we will show that $F(\ell) \sim F(a)$. Using Corollary~\ref{cor-Lflat}, take an interval $I \subseteq K$ with $\ell \in I^{K\langle \ell \rangle}$ such that $F(y) \sim F(\ell)$ for all $y \in I^{K\langle \ell \rangle}$. Then $F(a) \sim F(y) \sim F(\ell)$, since $a \in I^M$.
\end{proof}

\noindent
A \textbf{closed $v$-ball in $K$} is a set of the form $\big\{y \in K: v(y-a)\geq\gamma\big\}$ for some $a \in K$ and some $\gamma \in \Gamma_\infty$. Let $\cB$ be a collection of closed $v$-balls in $K$. Then $\cB$ is said to be \textbf{nested} if $B_1 \cap B_2 \neq \emptyset$ for any $B_1,B_2 \in \cB$. If every nested collection of closed $v$-balls in $K$ has nonempty intersection in $K$, then $K$ is said to be \textbf{spherically complete}. If $M$ is a $\TO$-extension of $K$, then we let $\cB^M$ denote the collection $\{B^M:B \in \cB\}$. Below we list some standard facts about spherical completeness.

\begin{lemma}\label{lem-basicvball}\
\begin{enumerate}
\item Let $\cB$ be a nested collection of closed $v$-balls in $K$ with empty intersection. By compactness, $K$ has a simple $\TO$-extension $K\langle a \rangle$ with $a \in \bigcap\cB^{K\langle a \rangle}$. 
\item Let $\cB$ and $a$ be as in (1). Then the set $v(a-K)$ has no largest element and for $y \in K$, the value $v(a-y)$ does not depend on the choice of $a$, just on the assumption $a \in \bigcap\cB^{K\langle a \rangle}$. 
\item Suppose $M$ is an immediate extension of $K$ and let $a\in M\setminus K$. Then the collection of all closed $v$-balls $B$ in $K$ with $a \in B^M$ is nested and has empty intersection in $K$.
\item If $K$ is spherically complete, then $K$ has no proper immediate $\TO$-extensions.
\end{enumerate}
\end{lemma}

\subsection{Power bounded theories}
A \textbf{power function on $K$} is an $\cL(K)$-definable endomorphism of the multiplicative group $K^>$. Each power function $F$ is $\cC^1$ on $K^>$ and uniquely determined by $F'(1)$, and we set
\[
\Lambda\ \coloneqq \ \big\{F'(1): F \text{ is a power function on }K\big\}.
\]
Then $\Lambda$ is a subfield of $K$, and it is called the \textbf{field of exponents of $K$}. For $a \in K^>$ and a power function $F$, we suggestively write $F(a)$ as $a^\lambda$ where $\lambda=F'(1)$. We say that $K$ is \textbf{power bounded} if for each $\cL(K)$-definable function $F\colon K\to K$, there is $\lambda$ in the field of exponents of $K$ with $|F(x)|<x^\lambda$ for all sufficiently large positive $x$.

\medskip\noindent
An \textbf{exponential function on $K$} is an ordered group isomorphism from the additive group $K$ to the multiplicative group $K^>$. Any exponential function on $K$ grows more quickly than every power function on $K$. By~\cite{Mi96}, either $K$ is power bounded or $K$ defines an exponential function. It follows that being power bounded is a property of the theory $T$ (we say that \emph{$T$ is power bounded}). If $T$ is power bounded, then each power function on $K$ is $\cL(\emptyset)$-definable, so we refer to the field of exponents $\Lambda$ as the \emph{field of exponents of $T$}, as $\Lambda$ does not depend on $K$. If $T$ is power bounded with archimedean field of exponents, then $T$ is said to be \textbf{polynomially bounded}.

\medskip\noindent
If $T$ is power bounded with field of exponents $\Lambda$, then the value group $\Gamma$ naturally admits the structure of an ordered $\Lambda$-vector space by setting $\lambda v(a) \coloneqq v(a^\lambda)$ for $a \in K^>$ (this does not depend on the choice of $a$). If $T$ is power bounded, then we have a better understanding of the immediate $\TO$-extensions of $K$ thanks to the following result of Tyne:

\begin{fact}[\cite{Ty03}, Theorems 12.10 and 13.4]
Let $T$ be power bounded and let $K\langle \ell \rangle$ be a simple $\TO$-extension of $K$.
\begin{enumerate}
\item If $\Gamma_{K\langle \ell \rangle} \neq \Gamma$, then there is $a \in K$ with $v(\ell-a) \not\in \Gamma$.
\item If $\res K\langle \ell \rangle \neq \res K$, then there are $a,b \in K$ with $b(\ell-a)\preceq 1$ and $\overbar{b(\ell-a)} \not\in \res K$.
\end{enumerate}
\end{fact}

\noindent
Item (1) above is often referred to as the \textbf{valuation property}, and item (2) is called the \textbf{residue property}. These properties allow us to characterize the simple immediate $\TO$-extensions of $K$ as follows:

\begin{lemma}
\label{lem-immediateext2}
Let $T$ be power bounded and let $K\langle \ell \rangle$ be a simple $\TO$-extension of $K$. The following are equivalent:
\begin{enumerate}
\item $v(\ell-K)$ has no largest element;
\item for each $a \in K$ there is $d \in K$ with $\ell-a \sim d$;
\item $K\langle \ell \rangle$ is an immediate extension of $K$.
\end{enumerate}
\end{lemma}
\begin{proof}
Assume (1) holds, let $a \in K$, and take $b \in K$ with $\ell-b \prec \ell - a$. Then $\ell-a \sim \ell - a - (\ell-b) = b-a\in K$. Now, assume (2) holds. Then $v(\ell-a) \in \Gamma$ for each $a \in K$, so $\Gamma_{K\langle \ell \rangle}=\Gamma$ by the valuation property. Let $a,b \in K$ with $b(\ell-a) \preceq 1$ and take $d \in K$ with $\ell-a \sim d$. If $b = 0$, then $\overbar{b(\ell-a)} =0 \in \res(K)$ and if $b \neq 0$, then $b(\ell-a) \sim bd$ so $\overbar{b(\ell-a)} =\overbar{bd} \in \res(K)$. Thus, $\res K\langle \ell \rangle = \res K$ by the residue property. Finally, suppose (3) holds, let $a \in K$, and take $d \in K$ with $\ell-a \sim d$. Then $v(\ell-a-d) >v(\ell-a)$, so $v(\ell-K)$ has no largest element. 
\end{proof}

\noindent
We can use this equivalence together with Lemma~\ref{lem-basicvball} to show that if $T$ is power bounded, then any nested collection of closed $v$-balls has nonempty intersection in an \emph{immediate} $\TO$-extension of $K$.

\begin{corollary}
\label{cor-immediateext3}
Let $T$ be power bounded and let $\cB$ be a nested collection of closed $v$-balls in $K$ with empty intersection. Then there is a simple immediate $\TO$-extension $K\langle \ell \rangle$ of $K$ with $\ell \in \bigcap\cB^{K\langle \ell \rangle}$. Given $a\in \bigcap \cB^M$, there is a unique $\LO(K)$-embedding $K\langle \ell \rangle\to M$ sending $\ell$ to $a$.
\end{corollary}
\begin{proof}
Using (1) of Lemma~\ref{lem-basicvball}, let $K\langle \ell\rangle$ be a simple $\TO$-extension of $K$ with $\ell \in \bigcap\cB^{K\langle \ell \rangle}$. By (2) of Lemma~\ref{lem-basicvball}, the set $v(\ell-K)$ has no largest element, so $K\langle \ell\rangle$ is an immediate extension of $K$ by Lemma~\ref{lem-immediateext2}. For $a\in \bigcap \cB^M$, we have $v(\ell-y) = v(a-y)$ for all $y \in K$ by (2) of Lemma~\ref{lem-basicvball}, so Corollary~\ref{cor-immediateext1} gives us a unique $\LO(K)$-embedding $K\langle \ell \rangle\to M$ sending $\ell$ to $a$.
\end{proof}

\begin{corollary}
\label{cor-immediateext}
Suppose $T$ is power bounded. Then $K$ has a spherically complete immediate $\TO$-extension which is unique up to $\LO(K)$-isomorphism. 
\end{corollary}
\begin{proof}
If $K$ is not itself spherically complete, then $K$ has a proper immediate $\TO$-extension by Corollary~\ref{cor-immediateext3}. It follows by Zorn's lemma that $K$ has a spherically complete immediate $\TO$-extension. For uniqueness, let $L$ and $M$ be two spherically complete immediate $\TO$-extensions of $K$. We first show that there is an $\LO(K)$-embedding $L \to M$. For this, we assume $K \neq L$, and we let $\ell \in L \setminus K$. Let $\cB$ be the collection of all closed $v$-balls $B$ in $K$ with $\ell \in B^L$. This collection is nested and has empty intersection in $K$ by (3) of Lemma~\ref{lem-basicvball}. Let $a \in\bigcap \cB^M$ and, again using Corollary~\ref{cor-immediateext3}, take an $\LO(K)$-embedding $K\langle \ell\rangle \to M$ sending $\ell$ to $a$. Continuing in this manner, we construct an $\LO(K)$-embedding $L \to M$, and we identify $L$ with an $\LO$-substructure of $M$ via this embedding. Then $M$ is an immediate extension of $L$, so $L = M$ by (4) of Lemma~\ref{lem-basicvball}, as $L$ is spherically complete.
\end{proof}

\noindent
The assumption of power boundedness in Corollary~\ref{cor-immediateext} is necessary. 

\begin{remark}
\label{rem-noimmediateext}
If $T$ is not power bounded, then $K$ has no spherically complete $\TO$-extension. To see this, we use Miller's dichotomy and a theorem of Kuhlmann, Kuhlmann, and Shelah. Suppose toward contradiction that $K$ is itself spherically complete and that $T$ is not power bounded. By~\cite{Mi96}, $K$ admits an $\cL(\emptyset)$-definable exponential function $\exp$. Let $\k\subseteq \cO$ be a lift of $\res(K)$, so $\exp\!|_{\k}$ is an exponential function on $\k$. Using~\cite[Lemma 3.3.32]{ADH17}, we take a subgroup $\mfM \subseteq K^>$ such that $v|_{\mfM}\colon \mfM\to \Gamma$ is a group isomorphism. Using~\cite[Corollary 3.3.42]{ADH17}, we get an ordered valued field isomorphism from $K$ to the ordered Hahn field $\k[[\mfM]]$ which is the identity on $\mfM$ and $\k$. The exponential on $K$ induces an exponential on $\k[[\mfM]]$ which restricts to the exponential $\exp\!|_{\k}$ on $\k$, contradicting the main theorem in~\cite{KKS97}.
\end{remark}

\subsection{Coarsening and specialization} \label{subsec-TOcoarsening}
In this subsection, we set up notation and prove some basic lemmas about coarsening and specialization. For the remainder of this subsection, we assume that $T$ is power bounded with field of exponents $\Lambda$, and we let $\Delta$ be a nontrivial convex $\Lambda$-subspace of $\Gamma$. We set $\dot{\Gamma}\ \coloneqq \ \Gamma/\Delta$, and we let $\dot{v}\colon K^\times \to \dot{\Gamma}$ be the map $a\mapsto va+\Delta \in\dot{\Gamma}$. Then $\dot{v}$ is a Krull valuation on $K$ with valuation ring and maximal ideal
\[
\dot{\cO}\ \coloneqq \ \{y \in K: vy\geq\delta\text{ for some }\delta \in \Delta\},\qquad \dot{\smallo}\ \coloneqq \ \{y \in K: vy>\Delta\}.
\]

\begin{lemma}
\label{lem-CoarseTconvex}
The valuation ring $\dot{\cO}$ is a $T$-convex valuation ring of $K$.
\end{lemma}
\begin{proof}
Let $F\colon K\to K$ be a continuous $\cL(\emptyset)$-definable function and let $a \in \dot{\cO}$. We need to show that $F(a)\in \dot{\cO}$. If $va \geq 0$, then $a \in \cO$, so $F(a) \in \cO\subseteq \dot{\cO}$. Suppose $va <0$, so $va \in \Delta^<$. Take $\lambda \in \Lambda$ and $b \in K^>$ such that $|F(y)|<|y|^\lambda$ for all $y\in K$ with $|y|>b$. As $F$ is $\cL(\emptyset)$-definable, we may assume $\{b\}$ is $\cL(\emptyset)$-definable as well, so $vb = 0$ and $|a|> b$. Then $|F(a)|<|a|^\lambda$, so $vF(a)\geq \lambda va \in \Delta$.
\end{proof}

\noindent
We let $K_\Delta$ denote the $\LO$-structure $(K,\dot{\cO})$, so $K_\Delta \models \TO$ by the above lemma. We refer to $K_\Delta$ as the \textbf{$\Delta$-coarsening of $K$}. The residue field $\res(K_\Delta)= \dot{\cO}/\dot{\smallo}$ is itself a $\TO$-model with value group $\Delta$ and 
$T$-convex valuation ring $\cO_{\res(K_\Delta)} = \{a+\dot{\smallo}:a \in \cO\}$. We refer to $\res(K_\Delta)$ with this valuation ring as the \textbf{$\Delta$-specialization of $K$}. Note that $\cO_{\res(K_\Delta)}/\smallo_{\res(K_\Delta)}$ is naturally $\cL$-isomorphic to $\res(K)$.


\begin{fact}[\cite{ADH17}, Corollary 3.4.6]
\label{fact-coarsespeccomp}
$K$ is spherically complete if and only if $K_\Delta$ and $\res(K_\Delta)$ are both spherically complete.
\end{fact}

\noindent
Let $M$ be a $\TO$-extension of $K_\Delta$ with $\Gamma_M = \dot{\Gamma}$. Let $\cO_{\res(M)}$ be a $T$-convex valuation ring of $\res(M)$ and suppose that the expansion of $\res(M)$ by $\cO_{\res(M)}$ is a $\TO$-extension of $\res(K_\Delta)$ with $\Gamma_{\res(M)}=\Delta$. Let $\cO_M^*\subseteq M$ be the convex subring
\[
\cO_M^*\ \coloneqq \ \{a \in M: a \in \cO_M\text{ and }\bar{a} \in \cO_{\res(M)}\}.
\]

\begin{lemma}
The convex subring $\cO_M^*$ is a $T$-convex valuation ring of $M$ and $\cO_M^*\cap K = \cO$.
\end{lemma}
\begin{proof}
Let $a \in \cO_M^*$ and let $F\colon M\to M$ be an $\cL(\emptyset)$-definable continuous function. Since $\bar{a} \in \cO_{\res(M)}$ and $\cO_{\res(M)}$ is $T$-convex, we have $\overbar{F(a)} = F(\bar{a}) \in \cO_{\res(M)}$ by~\cite[Lemma 1.13]{DL95}. Thus, $F(a) \in \cO_M^*$, so $\cO_M^*$ is $T$-convex. The equality $\cO_M^*\cap K = \cO$ follows from the equivalence
\[
y \in \cO \ \Longleftrightarrow \ y\in \dot{\cO}\text{ and }y+\dot{\smallo} \in \cO_{\res(K_\Delta)}
\]
for $y \in K$.
\end{proof}

\noindent
Let $M^*$ be the $\TO$-model with underlying $T$-model $M$ and $T$-convex valuation ring $\cO_{M^*} = \cO_M^*$, as defined above. Then $M^*$ is a $\TO$-extension of $K$, and we have $M^*_\Delta = M$ (as $\TO$-models).

\begin{lemma}
$\Gamma_{M^*} = \Gamma$ and $\res(M^*)$ is naturally $\cL$-isomorphic to $\cO_{\res(M)}/\smallo_{\res(M)}$.
\end{lemma}
\begin{proof}
As $M^*$ is a $\TO$-extension of $K$, we have $\Gamma\subseteq \Gamma_{M^*}$. For the other inclusion, let $a \in (M^*)^\times$. We need to find $b \in K^\times$ with $ab\inv \in \cO_{M^*}^\times$. First, take $f \in K_\Delta^\times$ with $af\inv \in \cO_M^\times$ and set $u \coloneqq \overbar{af\inv} \in \res(M)^\times$. Next, take $g \in \dot{\cO}^\times$ with $u\bar{g}\inv \in \cO_{\res(M)}^\times$. Then for $b \coloneqq fg$, we have $ab\inv \in \cO_{M^*}^\times$, as desired. As for the residue field, note that $\{\bar{a}:a \in \cO_{M^*}\} = \cO_{\res(M)}$, so we have a surjection
\[
a\mapsto \bar{a} + \smallo_{\res(M)}\colon \cO_{M^*}\to \cO_{\res(M)}/\smallo_{\res(M)}
\]
with kernel $\smallo_{M^*}$. This induces the desired isomorphism $\res(M^*)\to \cO_{\res(M)}/\smallo_{\res(M)}$.
\end{proof}

\noindent
By the above lemma, we see that $M^*$ is an immediate extension of $K$ if and only if $\res(M)$ is an immediate extension of $\res(K_\Delta)$. We summarize the discussion above with a diagram:
\[
\begin{tikzcd}
K_\Delta \arrow[r,"\LO" '] \arrow[d]& M \arrow[d]&\hspace{-1.7cm}=M^*_\Delta \\
\res(K_\Delta)\arrow[r,"\LO" '] \arrow[d]& \res(M)\arrow[d]\\
\res(K) \arrow[r,"\cL" '] &\res(M^*)
\end{tikzcd}
\]
Horizontal arrows are all embeddings in the indicated language. Downward arrows are projections\footnote{Often called \emph{places} in the literature.} and are only defined on the $T$-convex valuation ring of their source. Every square commutes.

\section{$T$-derivations}\label{sec-Tderiv}
\noindent
In this section, we fix a map $\der\colon K \to K$. For $a \in K$, we use $a'$ or $\der a$ in place of $\der(a)$, and we use $a^{(r)}$ in place of $\der^r(a)$. We define the \textbf{jets of $a$} as follows:
\[
\jet^r(a)\ \coloneqq \ (a, a',\ldots, a^{(r)}),\qquad \jet^\infty(a)\ \coloneqq\ (a^{(n)})_{n \in \N}.
\]
Occasionally, we omit the parentheses and write $\jet^r a$ or $\jet^\infty a$. To make some statements cleaner, we let $\jet^{-1}(a) \coloneqq 0 \in K^0$. Given a tuple $b = (b_1, \ldots, b_n) \in K^n$, we use $\der b$ or $b'$ to denote the tuple $(b_1', \ldots, b_n')$. We also let $\jet^r(b) \coloneqq (\jet^rb_1,\ldots,\jet^rb_n) \in K^{n(1+r)}$. Given a set $B \subseteq K^n$, we let
\[
\der B\ \coloneqq \ \{b':b \in B\},\qquad \jet^r(B)\ \coloneqq \ \big\{\jet^r(b):b \in B\big\}.
\]

\medskip\noindent
Given an $\cL(\emptyset)$-definable $\cC^1$-function $F\colon U \to K$ with $U \subseteq K^n$ open, we say that $\der$ is \textbf{compatible with $F$} if
\[
F(u)'\ =\ \nabla F(u) \cdot u'
\]
for each $u\in U$. We say that $\der$ is a \textbf{$T$-derivation on $K$} if $\der$ is compatible with every $\cL(\emptyset)$-definable $\cC^1$-function with open domain.

\medskip\noindent
The basic properties of $T$-derivations were systematically studied in~\cite{FK19}. Compatibility with the functions $(x,y) \mapsto x+y$ and $(x,y) \mapsto xy$ gives that any $T$-derivation on $K$ is a derivation on $K$. Let $\Ld \coloneqq \cL \cup \{ \der\}$, and let $\Td$ be the $\Ld$-theory which extends $T$ by axioms asserting that $\der$ is a $T$-derivation. For the rest of this section, let $K = (K,\der) \models \Td$. We let $C \coloneqq \ker(\der)$ denote the constant field of $K$. By~\cite[Lemma 2.3]{FK19}, the constant field $C$ is the underlying set of an elementary $\cL$-substructure of $K$. We recall three facts from~\cite{FK19} for later use:

\begin{fact}[\cite{FK19}, Lemma 2.12]
\label{fact-basicTderivation2}
Let $U\subseteq K^n$ be open and let $F\colon U\to K$ be an $\cL(K)$-definable $\cC^1$-function. Then there is a (necessarily unique) $\cL(K)$-definable function $F^{[\der]}\colon U \to K$ such that
\[
F(u)'\ =\ F^{[\der]}(u)+\nabla F(u)\cdot u'
\]
for all $u \in U$.
\end{fact}

\begin{fact}[\cite{FK19}, Lemma 2.13]
\label{fact-transext}
Let $M$ be a $T$-extension of $K$, let $A$ be a basis for $M$ over $K$, and let $s\colon A\to M$ be a map. There is a unique extension of $\der$ to a $T$-derivation on $M$ such that $a'= s(a)$ for all $a \in A$.
\end{fact}

\subsection{Affine conjugation}
\noindent
In this subsection, fix $r\geq 0$ and let $F\colon K^{1+r}\to K$ be an $\cL(K)$-definable function. For $k =0,\ldots, r$, we identify each variable $Y_k$ with the $k^{\text{\tiny th}}$ coordinate function $K^{1+r}\to K$. We let $Y= (Y_0,\ldots,Y_r)$, so $Y\colon K^{1+r}\to K^{1+r}$ is the identity map.

\begin{definition}
$F$ is said to be in \textbf{implicit form} if
\[
F\ =\ \fm_F\big(Y_r- I_F(Y_0,\ldots,Y_{r-1})\big)
\]
for some $\fm_F \in K^\times$ and some $\cL(K)$-definable function $I_F\colon K^r\to K$.
\end{definition}

\noindent
If $F$ is in implicit form, then $F(a,b) = 0$ if and only if $b = I_F(a)$ for $a\in K^r$ and $b \in K$. Thus, $I_F$ is an \emph{implicit function} for $F$. This is the source of the name ``implicit form'' and the notation $I_F$. By our convention for nullary functions, the unary functions in implicit form are exactly the functions of the form $\fm(Y_0- d)$ where $\fm \in K^\times$ and $d \in K$. Often, we omit the variables $Y_0,\ldots,Y_{r-1}$ and just write $F = \fm_F(Y_r-I_F)$ for $F$ in implicit form. 

\medskip\noindent
We may associate to $F$ the unary $\Ld(K)$-definable function $y \mapsto F(\jet^r y)$. For $k\leq r$ and $y \in K$, we have $Y_k(\jet^ry)= y^{(k)}$. As is the case with differential polynomials, these functions $y \mapsto F(\jet^r y)$ can be additively and multiplicatively conjugated (an operation we call \emph{affine conjugation}). They can also be compositionally conjugated, as we will see in Subsection~\ref{subsec-compconj}.

\medskip\noindent
Fix $a \in K$ and $d \in K^\times$. We let $Y^\der_{+a,\times d} = \big((Y_0)^\der_{+a,\times d},\ldots,(Y_r)^\der_{+a,\times d}\big)\colon K^{1+r}\to K^{1+r}$ be the map with coordinate functions
\[
(Y_k)^\der_{+a,\times d}\ \coloneqq \ a^{(k)}+\sum_{i=0}^k \binom{k}{i}d^{(k-i)}Y_i,\qquad k = 0,\ldots,r.
\]
Then $Y^\der_{+a,\times d}$ is a bijective $K$-affine map, and $Y^\der_{+a,\times d}(\jet^r y) = \jet^r(dy+a)$ for $y \in K$. We let $F^\der_{+a,\times d} \coloneqq F\circ Y^\der_{+a,\times d}$, so $F^\der_{+a,\times d}$ is $\cL(K)$-definable and
\[
F^\der_{+a,\times d}(\jet^ry)\ =\ F\big(\jet^r(dy+a)\big)
\]
for each $y \in K$. When $\der$ is clear from context, we drop the superscript and just write $F_{+a,\times d}$. For notational simplicity, we use $F_{+a}$ in place of $F_{+a,\times 1}$ and $F_{\times d}$ in place of $F_{+0,\times d}$. We also set $F_{-a,\times d} \coloneqq F_{+(-a),\times d}$. A straightforward computation gives us the following:

\begin{lemma}
\label{lem-addmultconj}
Suppose $F$ is in implicit form. Then $F_{+a,\times d}$ is also in implicit form with
\[
\fm_{F_{+a,\times d}} \ =\ d\fm_F,\qquad I_{F_{+a,\times d}}\ =\ d\inv\Big((I_F)_{+a,\times d} - a^{(r)}-\sum_{i=0}^{r-1} \binom{r}{i}d^{(r-i)}Y_i\Big).
\]
\end{lemma}

\noindent
Given an $\cL(K)$-definable set $A\subseteq K^{1+r}$, we set
\[
A^\der_{+a,\times d}\ \coloneqq \ \big\{u \in K^{1+r}:Y^\der_{+a,\times d}(u)\in A\big\},
\]
so $A^\der_{+a,\times d}$ is $\cL(K)$-definable, $\dimL(A^\der_{+a,\times d}) = \dimL(A)$, and $\jet^r(y) \in A^\der_{+a,\times d} \Longleftrightarrow \jet^r(dy+a) \in A$ for $y \in K$. Again, we drop the superscript if $\der$ is clear from context. 
\subsection{Compositional conjugation}\label{subsec-compconj}
In this subsection, $r\geq 0$ is again fixed and $F\colon K^{1+r}\to K$ is an $\cL(K)$-definable function. We continue to identify $Y= (Y_0,\ldots,Y_r)$ with the identity map and each $Y_k$ with the $k^{\text{\tiny th}}$ coordinate function.

\medskip\noindent
Fix $\phi \in K^\times$. Then $\derdelta \coloneqq \phi\inv\der$ is also a $T$-derivation on $K$, since for each $\cL(\0)$-definable $\cC^1$-function $G\colon U \to K$ with $U \subseteq K^n$ open and for each $u \in U$, we have
\[
\derdelta G(u)\ =\ \phi\inv\der G(u)\ =\ \phi\inv\big(\nabla G(u)\cdot \der u\big)\ =\ \nabla G(u) \cdot \derdelta u.
\]
We let $K^\phi = (K,\derdelta)$ be the expansion of the $\cL$-structure $K$ by the $T$-derivation $\derdelta$, and we refer to $K^\phi$ as the \textbf{compositional conjugate of $K$ by $\phi$}. For $\psi \in K^\times$, we have $(K^\phi)^\psi = K^{\phi\psi}$. 

\medskip\noindent
Let $n \geq 0$. Subsection 5.7 in~\cite{ADH17} gives for each $k\leq n$ an element $\xi^n_k(\phi)\in \Q[\phi,\der\phi,\ldots,\der^n\phi]$ such that
\[
\der^ny\ =\ \xi^n_0(\phi)y+ \xi^n_1(\phi)\derdelta y+\cdots+\xi^n_n(\phi)\derdelta^ny.
\]
In~\cite{ADH17}, $\xi^n_k(\phi)$ is instead called $F^n_k(\phi)$; we use different notation here and reserve $F$ for definable functions. The values of $\xi^n_k(\phi)$ are given by the recurrence relation:
\[
\xi^n_n(\phi) \ =\ \phi^n,\qquad \xi^n_0(\phi) \ =\ 0\ \text{ for }n>0,\qquad \xi^{n+1}_k(\phi)\ =\ \der\xi^n_k(\phi)+ \phi\xi^n_{k-1}(\phi)\ \text{ for }0<k\leq n.
\]
Let $Y^\phi_\der$ be the $K$-linear map $K^{1+r}\to K^{1+r}$ with matrix
\[
\begin{pmatrix}
\xi_0^0(\phi) & 0& 0 & \cdots & 0\\
\xi_0^1(\phi) & \xi_1^1(\phi) & 0 & \cdots & 0\\
\xi_0^2(\phi) & \xi_1^2(\phi) & \xi_2^2(\phi) & \cdots &0\\
\vdots & \vdots & \vdots & \ddots & \vdots\\
\xi_0^r(\phi) & \xi_1^r(\phi) & \xi_2^r(\phi) & \cdots & \xi_r^r(\phi)
\end{pmatrix}.
\]
Then $Y^\phi_\der$ is bijective and $Y^\phi_\der(\deltajet^r y) =\jet^r(y)$ for each $y \in K$.

\begin{lemma}\label{lem-prod}
$Y^1_\der = Y$ and $Y^\phi_\der \circ Y^\psi_\derdelta = Y^{\phi\psi}_\der$ for $\psi \in K^\times$. Given $a \in K$ and $d \in K^\times$, we have $Y^\der_{+a,\times d}\circ Y_\der^\phi =Y_\der^\phi\circ Y^\derdelta_{+a,\times d}$.
\end{lemma}
\begin{proof}
We will prove the last identity; the other identities can be established using a similar argument. Fix $a \in K$ and $d \in K^\times$. Since $Y^\der_{+a,\times d}\circ Y_\der^\phi$ and $Y_\der^\phi\circ Y^\derdelta_{+a,\times d}$ are both $K$-affine maps, the identity $Y^\der_{+a,\times d}\circ Y_\der^\phi =Y_\der^\phi\circ Y^\derdelta_{+a,\times d}$ holds on an affine subspace of $K^{1+r}$. Thus, it suffices to show that this identity holds on a dense subset of $K^{1+r}$. Let $U \subset K^{1+r}$ be $\cL(K)$-definable and open. By applying~\cite[Lemma 4.2]{FK19} to $K^\phi$, we find a $\Td$-extension $M = K\langle \deltajet^r b\rangle$ of $K^\phi$ with $\deltajet^r(b)\in U^M$. We have
\[
Y^\der_{+a,\times d}\big(Y_\der^\phi(\deltajet^r b)\big)\ =\ Y^\der_{+a,\times d}(\jet^r b)\ =\ \jet^r(db+a)\ =\ Y_\der^\phi\big(\deltajet^r(db+a)\big) \ =\ Y_\der^\phi\big(Y^\derdelta_{+a,\times d}(\deltajet^r b)\big).
\]
Since $Y^\der_{+a,\times d}\circ Y_\der^\phi$ and $Y_\der^\phi\circ Y^\derdelta_{+a,\times d}$ are $\cL(K)$-definable and $K$ is an elementary $\cL$-substructure of $M$, there is a tuple $u \in U$ with $Y^\der_{+a,\times d}\big(Y_\der^\phi(u)\big)= Y_\der^\phi\big(Y^\derdelta_{+a,\times d}(u)\big)$.
\end{proof}

\noindent
We set $F^\phi_\der \coloneqq F\circ Y^\phi_\der$, so $F^\phi_\der$ is $\cL(K)$-definable and $F^\phi_\der(\deltajet^ry) = F(\jet^ry)$ for all $y \in K$. Lemma~\ref{lem-prod} tells us that $F^1_\der = F$ and that
\[
(F^\phi_\der)^\psi_\derdelta \ =\ F^{\phi\psi}_\der\ \text{ for }\psi \in K^\times,\qquad (F^\der_{+a,\times d})_\der^\phi\ =\ (F_\der^\phi)^\derdelta_{+a,\times d}\ \text{ for $a \in K$ and $d \in K^\times$}.
\]
When $\der$ is clear from context, we write $F^\phi$ in place of $F^\phi_\der$, and we write $F^\phi_{+a,\times d}$ in place of $(F^\der_{+a,\times d})_\der^\phi$. Then $F^\phi_{+a,\times d}(\deltajet^r y)=F\big(\jet^r(dy+a)\big)$ for $y$ in $K$.
As with affine conjugation, we have the following:

\begin{lemma}
\label{lem-compconj}
Suppose $F$ is in implicit form. Then $F^\phi$ is also in implicit form with
\[
\fm_{F^\phi} \ =\ \phi^r\fm_F,\qquad I_{F^\phi}\ =\ \phi^{-r}\Big(I_F^\phi - \sum_{i=0}^{r-1} \xi^r_i(\phi)Y_i\Big).
\]
\end{lemma}

\noindent
Given an $\cL(K)$-definable set $A\subseteq K^{1+r}$, we set
\[
A_\der^\phi\ \coloneqq \ \big\{u \in K^{1+r}:Y^\phi_\der(u)\in A\big\},
\]
so $A_\der^\phi$ is $\cL(K)$-definable, $\dimL(A_\der^\phi) = \dimL(A)$, and $\deltajet^r(y) \in A_\der^\phi \Longleftrightarrow \jet^r(y) \in A$ for $y \in K$. As with definable functions, we drop the subscript and write $A^\phi$ when $\der$ is clear from context.

\subsection{Thin sets}\label{subsec-thin}
A subset $Z\subseteq K^n$ is said to be \textbf{thin} if $\jet^r(Z) \subseteq A$ for some $r$ and some $\cL(K)$-definable set $A \subseteq K^{n(1+r)}$ with $\dimL(A) < n(1+r)$. We note that if $Z \subseteq K$ is $\Ld(K)$-definable, then $Z$ is \emph{not} thin if and only if there is an elementary $\Td$-extension of $M$ of $K$ and an element $a \in Z^M$ such that $\mathscr{J}_\der^\infty(a)$ is $\cL(K)$-independent. The union of any two thin sets is thin, any subset of a thin set is thin, and the singleton $\{a\}$ is thin for $a \in K^n$. The constant field $C$ of $K$ is thin, since
\[
\jet^1(C)\ =\ \big\{(c,0):c \in C\big\}\ \subseteq\ K\times \{0\},
\]
and $K\times\{0\}$ is a 1-dimensional subset of $K^2$. Thus, $K$ is thin if $\der$ is trivial. Here is a converse:

\begin{proposition}
\label{prop-nondefinable}
Suppose that $\der$ is nontrivial. If $U \subseteq K^n$ is open, then $U$ is not thin.
\end{proposition}

\noindent
The case where $U = K$ plays a key part in the proof of Lemma~\ref{lem-nontrivialbound} below, but we prove Proposition~\ref{prop-nondefinable} for arbitrary open $U$ with an eye toward future applications. First, we need a short lemma.

\begin{lemma}
\label{lem-vanishonC}
Let $A \subseteq K^n$ be $\cL(K)$-definable and let $A_1,\ldots,A_n \subseteq K$ be infinite sets. If $A_1 \times \cdots \times A_n \subseteq A$, then $\dimL(A) = n$.
\end{lemma}
\begin{proof}
We proceed by induction on $n$, with the $n = 1$ case being clear. Let $n >1$, let $A \subseteq K^n$ be $\cL(K)$-definable, and let $A_1,\ldots,A_n \subseteq K$ be infinite sets with $A_1 \times \cdots \times A_n \subseteq A$. By~\cite[Proposition 4.1.5]{vdD98}, the set
\[
A^*\ \coloneqq \ \big\{y \in \pi_1(A): \dimL(A_y) = n-1\big\}
\]
is $\cL(K)$-definable and $\dimL(A) = \dimL(A^*)+(n-1)$, so we need to show that $\dimL(A^*) = 1$. If $y \in A_1$, then $A_2 \times \cdots \times A_n \subseteq A_y$, so $y \in A^*$ by our induction hypothesis. Thus, $A_1 \subseteq A^*$, so $\dimL(A^*) = 1$ by the $n = 1$ case.
\end{proof}

\begin{proof}[Proof of Proposition~\ref{prop-nondefinable}]
We begin with the case $n = 1$. Let $I \subseteq K$ be an open interval, let $A \subseteq K^{1+r}$ be $\cL(K)$-definable, and suppose $\jet^r(y)\in A$ for each $y \in I$. We will show that $\dimL(A) = 1+r$. Take $a\in K$ and $d \in K^>$ with $I = (a-d,a+d)$. By replacing $A$ with $A_{+a,\times d}$, we arrange that $I = (-1,1)$. Since $\der$ is nontrivial, we have $x \in K$ with $x' \neq 0$. By inverting $x$ if need be, we arrange that $-1<x<1$, so 
\[
c_0+c_1x+\frac{1}{2}c_2x^2+\cdots+\frac{1}{r!}c_rx^r \in I
\]
for any constants $c_0,\ldots,c_r \in [0,1/3]_C$. By replacing $K$ and $A$ with $K^\phi$ and $A^\phi$ for $\phi \coloneqq x'$, we arrange that $x' = 1$. Let $P$ be the matrix
\[
\begin{pmatrix}
1& x& \frac{1}{2}x^2 & \cdots &\frac{1}{r!}x^r\\
0 & 1& x & \cdots & \frac{1}{(r-1)!}x^{r-1}\\
0 & 0 & 1 & \cdots & \frac{1}{(r-2)!}x^{r-2}\\
\vdots & \vdots & \vdots & \ddots & \vdots\\
0 & 0&0& \cdots & 1
\end{pmatrix}.
\]
Then $P$ is invertible and for each $c = (c_0,\ldots,c_r) \in [0,1/3]_C^{1+r}$, we have
\[
Pc\ =\ \begin{pmatrix}
c_0+c_1x+\frac{1}{2}c_2x^2+\cdots+\frac{1}{r!}c_rx^r\\
c_1+c_2x+\cdots+\frac{1}{(r-1)!}c_rx^{r-1}\\
c_2+\cdots+\frac{1}{(r-2)!}c_rx^{r-2}\\
\vdots\\
c_r
\end{pmatrix}\ =\ \jet^r\Big(c_0+c_1x+\frac{1}{2}c_2x^2+\cdots+\frac{1}{r!}c_rx^r\Big)\ \in \ A.
\]
Thus, $[0,1/3]_C^{1+r}\subseteq P\inv A \coloneqq \{P\inv y :y \in A\}$. Lemma~\ref{lem-vanishonC} (with $P\inv A$ in place of $A$ and $[0,1/3]_C$ in place of each $A_i$) gives
\[
\dimL(A)\ =\ \dimL(P\inv A) \ =\ 1+r.
\]

We now handle the general case by induction on $n$. Let $n>1$, let $U \subseteq K^n$ be open, let $A \subseteq K^{n(1+r)}$ be $\cL(K)$-definable, and suppose $\jet^r(U)\subseteq A$. We need to show that $\dimL(A) = n(1+r)$. By~\cite[Proposition 4.1.5]{vdD98}, the set
\[
A^*\ \coloneqq \ \big\{y \in \pi_{1+r}(A): \dimL(A_y) = (n-1)(1+r)\big\}
\]
is $\cL(K)$-definable and $\dimL(A)\geq \dimL(A^*) + (n-1)(1+r)$, so it suffices to show that $\dimL(A^*) = 1+r$. For each $y \in \pi_1(U)$ and each $z \in U_y$, we have $\jet^r(z) \in A_{\mathscr{J}_\der^r(y)}$, so $\dimL(A_{\mathscr{J}_\der^r(y)}) = (n-1)(1+r)$ by our induction hypothesis. Thus, $\jet^r(y) \in A^*$ for each $y \in \pi_1(U)$, so $\dimL(A^*) = 1+r$ by the $n = 1$ case.
\end{proof}

\section{$T$-convex $T$-differential fields}\label{sec-TdO}
\noindent
Let $\LdO \coloneqq \LO\cup \Ld= \cL\cup\{\cO,\der\}$. As defined in the introduction, a \textbf{$T$-convex $T$-differential field} is an $\LdO$-structure $K = (K,\cO,\der)$ such that
\begin{enumerate}
\item $(K,\cO) \models \TO$;
\item $(K,\der) \models \Td$;
\item $\der$ is continuous with respect to the valuation topology.
\end{enumerate}
Let $\TdO$ be the $\LdO$-theory of $T$-convex $T$-differential fields. For the remainder of this article, let $K = (K,\cO,\der) \models \TdO$. Since $\cO$ is a proper $T$-convex subring of $K$, the valuation topology and the order topology coincide by~\cite[Lemma 2.4.1]{ADH17}. If $\der\smallo \subseteq \smallo$, then $K$ is said to have \textbf{small derivation}. Small derivation implies continuity of the derivation. The following fact demonstrates how $\der\cO$ is controlled by $\der\smallo$.

\begin{fact}[\cite{ADH17}, Lemma 4.4.2]
\label{fact-bigtobig}
If $K$ has small derivation, then $\der \cO \subseteq \cO$. Consequently, $\der\smallo \subseteq\phi \smallo \Longrightarrow \der\cO \subseteq \phi \cO$ for each $\phi \in K^\times$.
\end{fact}

\noindent
Suppose $K$ has small derivation. By the above fact, $\der$ induces a map $\bar{a} \mapsto \overbar{\der a}\colon \res(K)\to \res(K)$. We denote this map also by $\der$, and we call it the \textbf{induced derivation on $\res(K)$}. This induced derivation is even a $T$-derivation. To see this, let $F$ be an $n$-ary $\cL(\emptyset)$-definable $\cC^1$-function with open domain and let $\varphi$ be the $\cL(\emptyset)$-formula defining the domain of $F$. Set
\[
U\ \coloneqq \ \varphi(K)\ \subseteq\ K^n,\qquad V\ \coloneqq \ \varphi(\res K)\ \subseteq\ \res(K)^n,
\]
so $V \subseteq \overbar{U}$. Let $F$ denote both its interpretation as a function $U \to K$ and its interpretation as a function $V \to \res(K)$ and let $u \in U$ with $\bar{u}\in V$. By~\cite[Lemma 1.13]{DL95}, we have $\overbar{F(u)} = F(\bar{u})$, so 
\[
\der F(\bar{u})\ =\ \der\overbar{F(u)}\ =\ \overbar{\der F(u)}\ =\ \overbar{\nabla F(u)\cdot\der u}\ =\ \nabla F(\bar{u})\cdot \der \bar{u}.
\]
Accordingly, we view $\res(K)$ as a model of $\Td$. Note that the induced derivation on $\res(K)$ is trivial if and only if $\der\cO \subseteq \smallo$.

\medskip\noindent
Let $\Delta$ be a nontrivial convex $\Lambda$-subspace of $\Gamma$. Recall the $\TO$-models $K_{\Delta}$ and $\res(K_{\Delta})$ associated to $K$ and $\Delta$ from Subsection~\ref{subsec-TOcoarsening}. We may view $K_{\Delta}$ as a $\TdO$-model with the same derivation as $K$; the valuation topology induced by $\dot{\cO}$ is either discrete or the same as the topology induced by $\cO$, so $\der$ is still continuous in $K_{\Delta}$. Suppose $K$ has small derivation. Then $K_{\Delta}$ does as well by~\cite[Corollary 4.4.4]{ADH17}, so we may consider the induced derivation on $\res(K_{\Delta})$. This induced derivation is a $T$-derivation by the remarks above, and it is also small, hence continuous. Thus, we may view $\res(K_{\Delta})$ as a $\TdO$-model as well. 

\subsection{Examples}
As mentioned in the introduction, the following example fits comfortably within the framework of~\cite{ADH18B}:

\begin{example}[Real closed valued differential fields]
\label{ex-rcvf}
Let $R$ be a real closed field and let $\cO$ be a proper convex valuation ring of $R$. Then $\cO$ is $\RCF$-convex by~\cite[Proposition 4.2]{DL95}. Let $\der$ be a derivation on $R$ which is continuous with respect to the valuation topology. By~\cite[Proposition 2.8]{FK19}, $\der$ is an $\RCF$-derivation, so $(R,\cO,\der) \models \RCF^{\cO,\der}$. Conversely, every model of $\RCF^{\cO,\der}$ is a real closed nontrivially valued field with a continuous derivation.
\end{example}

\noindent
Hardy fields, which are deeply connected to o-minimality, provide a rich source of new examples:

\begin{example}[$\cR$-Hardy fields]
\label{ex-hardy}
Let $\cR$ be an arbitrary o-minimal expansion of the real field in an appropriate language $\cL_{\cR}$. By extending $\cL_{\cR}$, we may assume that $\cL_{\cR}$ contains a constant symbol for each $r \in \R$ and that $T_{\cR} \coloneqq \Th(\cR)$ has quantifier elimination and a universal axiomatization. Following~\cite{DMM94}, we define an \textbf{$\cR$-Hardy field} to be a Hardy field $\cH$ which is closed under all function symbols in $\cL_{\cR}$. That is, $\cH$ is an ordered field of germs at $+\infty$ of unary functions $f\colon \R \to \R$ such that:
\begin{enumerate}
\item If the germ of $f$ belongs to $\cH$, then so does the germ of $f'$;
\item If $F$ is an $n$-ary function symbol in $\cL_{\cR}$ and the germs of $f_1,\ldots,f_n$ belong to $\cH$, then the germ of the composite function $x\mapsto F\big(f_1(x),\ldots,f_n(x)\big)$ belongs to $\cH$.
\end{enumerate}
Let $\cH$ be an $\cR$-Hardy field. Then $\cH$ admits a natural expansion to a model of $T_{\cR}$~\cite[Lemma 5.8]{DMM94}. Let $\der$ denote the natural derivation on $\cH$ which sends the germ of $f$ to the germ of $f'$, and let $\cO$ be the convex hull of $\R$ in $\cH$, where each $r \in \R$ is identified with the germ of the constant function $x\mapsto r$. Then $\cO$ is $T_{\cR}$-convex by Fact~\ref{fact-hull}, and the chain rule from elementary calculus tells us that $\der$ is a $T_{\cR}$-derivation. As the derivation on $\cH$ is small, it is continuous, so $\cH\models \TdO_{\cR}$ (assuming $\cO \neq \cH$).
\end{example}

\noindent
For a more algebraic example, we consider field of Puiseux series over $\R$:

\begin{example}[Puiseux series]\label{ex-PuiseuxTdO}
Let $\bigcup_{n>0} \R((t^{1/n}))$ be the field of Puiseux series over $\R$, ordered so that $0 < t<\R^>$. This field admits a canonical expansion to a model of $T_{\an}$, where each restricted analytic function is defined via Taylor expansion. The usual valuation ring on this field, the convex hull of $\R$, is $T_{\an}$-convex by Fact~\ref{fact-hull}. We define a derivation $\der$ on this field by setting
\[
\der\sum_{k = k_0}^\infty r_k t^{\frac{k}{n}}\ \coloneqq \ \sum_{k = k_0}^\infty \frac{kr_k}{n} t^{\frac{k}{n}-1}.
\]
Then $\der$ is easily seen to be compatible with all restricted analytic functions, so it is even a $T_{\an}$-derivation; see~\cite[Lemma 2.9]{FK19} for why this is sufficient. Since $\der$ is small, it is continuous, so the field of Puiseux series is a model of $\TdO_{\an}$.
\end{example}

\noindent
The Puiseux series sit inside of the much richer field of \emph{logarithmic-exponential transseries}, which serves as another motivating example:

\begin{example}[Transseries]\label{ex-transseriesTdO}
Let $\T$ be the ordered valued differential field of logarithmic-exponential transseries with valuation ring $\cO$ and derivation $\der$; see Appendix A of~\cite{ADH17} for a detailed definition. By~\cite[Corollary 2.8]{DMM97}, $\T$ admits a canonical expansion to a model of $T_{\an}$, which we denote by $\T_{\an}$. Since $\cO$ is the convex hull of the $T_{\an}$-model $\R_{\an}$, it is $T_{\an}$-convex by Fact~\ref{fact-hull}. By~\cite[Corollary 3.3]{DMM01}, the derivation $\der$ is compatible with all restricted analytic functions, so $\der$ is a $T_{\an}$-derivation by~\cite[Lemma 2.9]{FK19}. Since $\der$ is small, it is continuous, so $\T_{\an}\models \TdO_{\an}$.
\end{example}

\noindent
Examples~\ref{ex-hardy},~\ref{ex-PuiseuxTdO}, and~\ref{ex-transseriesTdO} above are all \emph{$H$-fields} (with constant field $\R$), as introduced in~\cite{AD02}. Models of $\TdO$ which are also $H$-fields are studied in more depth in~\cite{Ka22}. For an example which is not an $H$-field, consider the following:

\begin{example}[Ordered Hahn differential fields]
\label{ex-hahnTdO}
Let $\k = (\k,\der_{\k}) \models \Td_{\an}$ and let $\mfM$ be a divisible ordered abelian group, written multiplicatively. Consider the ordered Hahn field $\k[[\mfM]]$. We identify $\k$ with a subfield of $\k[[\mfM]]$ via the embedding $a \mapsto a\cdot 1_{\mfM}$. As with the field of Puiseux series, we expand $\k[[\mfM]]$ to a model of $T_{\an}$ using Taylor expansion; this was done when $\k = \R$ in~\cite{DMM94}, and the reader may consult~\cite[Proposition 2.13]{Ka21} for full details. We extend $\der_{\k}$ to a derivation $\der$ on $\k[[\mfM]]_{\an}$ by setting 
\[
\der\big(\sum_\fm f_\fm\fm\big)\ \coloneqq \ \sum_\fm\der_{\k}(f_\fm)\fm
\]
for $\sum_\fm f_\fm\fm \in \k[[\mfM]]_{\an}$. Using that $\der$ is strongly additive, one can verify that $\der$ is compatible with each restricted analytic function; see~\cite[Proposition 3.14]{Ka21} for full details. It follows by~\cite[Lemma 2.9]{FK19} that $\der$ is a $T_{\an}$-derivation. The constant field of $\k[[\mfM]]_{\an}$ is the subfield $C_{\k}[[\mfM]]$, where $C_{\k}$ is the constant field of $\k$. Let $\cO$ be the convex hull of $\k$ in $\k[[\mfM]]_{\an}$, so $\cO$ is $T_{\an}$-convex by Fact~\ref{fact-hull}. With respect to $\cO$, the derivation $\der$ is \textbf{monotone}, that is, $\der f\preceq f$ for all $f\in \k[[\mfM]]_{\an}$. In particular, $\der$ is small, so $\k[[\mfM]]_{\an}\models \TdO_{\an}$ (assuming $\mfM$ is nontrivial). This model is spherically complete and has \emph{many constants}: for each $f \in \k[[\mfM]]_{\an}$, there is a constant $c$ with $f\asymp c$.
\end{example}

\noindent
In~\cite{Sc00}, Scanlon proved an Ax-Kochen-Er\v{s}ov (AKE) result for Hahn differential fields $\k[[\mfM]]$ where $\k$ is an (unordered) differential field of characteristic zero and where the derivation on $\k$ is extended to $\k[[\mfM]]$ as in the example above. As a consequence, he showed that if $\mfM$ is divisible and $\k$ is differentially closed, then the valued differential field $\k[[\mfM]]$ is model complete. Scanlon's AKE result can also be used to show that if $\mfM$ is divisible and $\k$ is a closed ordered differential field (as introduced by Singer~\cite{Si78}), then the \emph{ordered} valued differential field $\k[[\mfM]]$ is model complete\footnote{In~\cite{Sc00}, it is assumed that $\k$ is closed under $n$-th roots for all $n>0$. This condition is not satisfied for even $n$ when $\k$ is real closed. However, this assumption can be removed; see~\cite[Section 6]{Sc03} and also~\cite[Theorem 8.0.3]{ADH17}. Thanks to Gabriel Ng for pointing this out.}. In~\cite{FK19}, it is shown that $\Td_{\an}$ has a model completion. This raises the question: if $\k$ is a model of this model completion and $\k[[\mfM]]_{\an}\models \TdO_{\an}$ is as above, then is $\k[[\mfM]]_{\an}$ model complete? As with the uniqueness questions discussed at the end of this article, answering this question likely requires some analog of \emph{differential henselianity}. For a generalization of Scanlon's AKE result, see~\cite{Ha18}.

\subsection{Strict extensions}
Let $M$ be a $\TdO$-extension of $K$. Recall from the introduction that $M$ is a \textbf{strict} extension of $K$ if
\[
\der \smallo \subseteq \phi\smallo\ \Longrightarrow\ \der_M\smallo_M \subseteq \phi\smallo_M\qquad\der \cO \subseteq \phi\smallo\ \Longrightarrow\ \der_M\cO_M \subseteq \phi\smallo_M
\]
for each $\phi \in K^\times$.

\medskip\noindent
We are interested in when $K$ has a spherically complete immediate strict $\TdO$-extension. If $\der$ is trivial and $T$ is power bounded, then $K$ has a spherically complete immediate $\TO$-extension $M$ by Corollary~\ref{cor-immediateext}. Viewed as a $\TdO$-model with trivial derivation, $M$ is a spherically complete immediate strict $\TdO$-extension of $K$. If $T$ is not power bounded, then $K$ has no spherically complete $\TO$-extension at all by Remark~\ref{rem-noimmediateext}. Accordingly, we make the following assumption:

\begin{assumption}\label{assp-standing}
For the remainder of this article, $T$ is power bounded with field of exponents $\Lambda$ and the derivation on $K$ is nontrivial.
\end{assumption}

\noindent
Below we list some basic but important facts about strict extensions.
\begin{enumerate}
\item $M$ is a strict $\TdO$-extension of $K$ if and only if $M^\phi$ is a strict $\TdO$-extension of $K^\phi$ for $\phi \in K^\times$.
\item If $M$ is a $\TdO$-extension of $K$ and $M$ is contained in a strict $\TdO$-extension of $K$, then $M$ is itself a strict extension of $K$.
\item If $M$ is a strict $\TdO$-extension of $L$ and $L$ is a strict $\TdO$-extension of $K$, then $M$ is a strict extension of $K$.
\item If $M$ is an elementary $\TdO$-extension of $K$, then $M$ is a strict extension of $K$.
\end{enumerate}

\medskip\noindent
In general, there may be $\TdO$-extensions $M$ of $K$ which are not strict but which nevertheless satisfy the condition $\der \smallo \subseteq \phi\smallo \Longrightarrow \der_M\smallo_M$ for all $\phi \in K^\times$. However, if $M$ is an \emph{immediate} extension of $K$, then this weaker condition implies strictness:

\begin{fact}[\cite{ADH18B}, Lemma 1.5]
\label{fact-onlychecksmall}
Suppose that $M$ is an immediate $\TO$-extension of $K$ and let $\der_M$ be a $T$-derivation on $M$ which extends $\der$. If
\[
\der \smallo \subseteq \phi\smallo\ \Longrightarrow \ \der_M \smallo_M \subseteq \phi \smallo_M
\]
for each $\phi\in K^\times$, then $M$ is a strict $\TdO$-extension of $K$.
\end{fact}

\noindent
Recall from the introduction the sets
\[
\Gamma(\der)\ \coloneqq \ \{v\phi:\der\smallo\subseteq \phi\smallo\},\qquad S(\der)\ \coloneqq \ \big\{\gamma \in \Gamma:\Gamma(\der) + \gamma = \Gamma(\der)\big\}.
\]
We sometimes write $\Gamma_K(\der)$ and $S_K(\der)$ if $K$ is not clear from context. Note that $\Gamma(\der)< v(\der\smallo)$ is a downward closed subset of $\Gamma$ and that $S(\der)$, the \emph{stabilizer} of $\Gamma(\der)$, is a convex subgroup of $\Gamma$. For $\phi \in K^\times$, we have
\[
\Gamma(\phi\inv\der) \ =\ \Gamma(\der)-v\phi,\qquad S(\phi\inv\der) \ =\ S(\der),
\]
so $S(\der)$ is invariant under compositional conjugation. If $M$ is a strict $\TdO$-extension of $K$ with $\Gamma_M = \Gamma$, then $\Gamma_M(\der) = \Gamma(\der)$ and $S_M(\der) = S(\der)$. Using that $K$ is real closed, we can show that $\Gamma(\der)$ is closed in $\Gamma$ (with respect to the order topology).

\begin{lemma}
\label{lem-supisactive}
Let $\alpha \in \Gamma$. If $\alpha - \epsilon \in \Gamma(\der)$ for each $\epsilon\in \Gamma^>$, then $\alpha \in \Gamma(\der)$. 
\end{lemma}
\begin{proof}
Suppose $\alpha \not\in \Gamma(\der)$. Take $a \in \smallo$ with $a>0$ and $v(a') \leq \alpha$. Then
\[
v\big((a^{1/2})'\big)\ =\ v(a')-\frac{1}{2}va\ \leq\ \alpha-\frac{1}{2}va,
\]
so $\alpha - \frac{1}{2}va \not\in \Gamma(\der)$.
\end{proof}

\section{Behavior of definable functions} \label{sec-behaviour}
\noindent
In this section, let $F\colon K^{1+r} \to K$ be an $\cL(K)$-definable function in implicit form. We set $vF \coloneqq v(\fm_F) \in \Gamma$, and we call $vF$ the \textbf{valuation of $F$}. Lemmas~\ref{lem-addmultconj} and~\ref{lem-compconj} give
\[
vF_{+a,\times d} \ =\ vF+vd,\qquad vF^\phi\ =\ vF+rv\phi
\]
for $a \in K$ and $d,\phi \in K^\times$. The valuation of $F$ acts as a sort of crude replacement for the Gaussian valuation associated to differential polynomials, used frequently throughout~\cite{ADH17}.

\subsection{Bounding $vF(\jet^ru)$ from above}
In this subsection, we use $vF$ to find points $u \in K$ where $F(\jet^ru)$ is ``not too small.''

\begin{lemma}
\label{lem-nontrivialbound}
Suppose that $K$ has small derivation and that the induced derivation on $\res(K)$ is nontrivial. Then $vF(\jet^ru)\leq vF$ for some $u \in \cO^\times$.
\end{lemma}
\begin{proof}
Fact~\ref{fact-resdef} gives that $\overbar{\Graph(I_F)}\subseteq \res(K)^{1+r}$ is $\cL(\res K)$-definable and $\dimL\!\big(\overbar{\Graph(I_F)}\big) \leq r$. Thus, the set
\[
\big\{y \in \res(K):\jet^r(y) \in \overbar{\Graph(I_F)}\big\}\cup \{0\}
\]
is a thin subset of $\res(K)$. Proposition~\ref{prop-nondefinable} applied to $\res(K)$ gives $u \in \cO^\times$ with $\jet^r(\bar{u}) = \overbar{\jet^r(u)} \not \in \overbar{\Graph(I_F)}$. Then either $I_F(\jet^{r-1}u) \succ 1$ or $I_F(\jet^{r-1}u) \preceq 1$ and 
\[
\overbar{u^{(r)}}\ \neq\ \overbar{I_F(\jet^{r-1}u)}.
\]
In either case, $u^{(r)}- I_F(\jet^{r-1}u)\succeq 1$, so
\[
F(\jet^ru) \ =\ \fm_F\big(u^{(r)}- I_F(\jet^{r-1}u)\big) \ \succeq \ \fm_F.\qedhere
\]
\end{proof}

\begin{lemma}
\label{lem-approxbound}
Suppose $S(\der) = \{0\}$ and let $\beta \in \Gamma^>$. Then there is $\gamma \in \Gamma(\der)$ and $u \in K$ with $|vu|<\beta$ such that
\[
v F(\jet^ru) \ \leq\ vF+ r\gamma+\beta.
\]
\end{lemma}
\begin{proof}
We claim that for any $\epsilon \in \Gamma^>$, we can find $\gamma \in \Gamma(\der)$ and $a \in \smallo$ such that $v(a')-\gamma\leq \epsilon$. Let $\epsilon \in \Gamma^>$ be given and, using that $\epsilon \not\in S(\der)= \{0\}$, take $\gamma \in \Gamma(\der)$ with $\gamma + \epsilon \not\in \Gamma(\der)$. Then there is $a \in \smallo$ with $v(a')\leq \gamma + \epsilon$, as desired. This claim yields an elementary $\TdO$-extension $M$ of $K$ with $\gamma \in \Gamma_M(\der)$ and $a \in \smallo_M$ such that 
\[
v(a')-\gamma\ <\ \Gamma^>.
\]
Let $\Delta$ be the convex $\Lambda$-subspace of $\Gamma_M$ consisting of all $\epsilon \in \Gamma_M$ with $|\epsilon|< \Gamma^>$ and let $\phi \in M^\times$ with $v\phi = \gamma$. Then $M^\phi$ has small derivation, so $M^\phi_\Delta$ does as well by~\cite[Corollary 4.4.4]{ADH17}. Since $v(\phi\inv a')< \Gamma^>$ and $v(a')> v\phi$, we have $\dot{v}(\phi\inv a')= 0$, so the derivation on $\res(M^\phi_\Delta)$ is nontrivial. Applying Lemma~\ref{lem-nontrivialbound} to $M_\Delta^\phi$ and $F^\phi$, we get $u \in M$ with $\dot{v}u= 0$ and
\[
\dot{v}F^\phi(\ojet{\phi\inv\der_M}^ru)\ \leq\ \dot{v}F^\phi. 
\]
Then $|vu|<\Gamma^>$ and
\[
vF^\phi(\ojet{\phi\inv\der_M}^ru)- vF^\phi\ =\ vF(\ojet{\der_M}^ru)-(vF+rv\phi)\ =\ vF(\ojet{\der_M}^ru)-vF-r\gamma \ <\ \Gamma^>.
\]
In particular, $|vu|<\beta$ and $vF(\ojet{\der_M}^ru)<vF+r\gamma+\beta$. As $M$ is an elementary $\TdO$-extension of $K$, the lemma follows. 
\end{proof}

\begin{corollary}
\label{cor-sameval}
Suppose $S(\der) = \{0\}$, let $\beta \in \Gamma^>$, and suppose that $vF(\jet^ra) = vF(\jet^rb)$ for all $a,b \in K$ with $|va|,|vb|<\beta$. Then there is $\gamma \in \Gamma(\der)$ such that $vF(\jet^ru)\leq vF+r\gamma$ for all $u\in K$ with $|vu|<\beta$.
\end{corollary}
\begin{proof}
We first handle the case $r = 0$, so $I_F\in K$. Take $a \in \cO^\times$ with $a \not\sim I_F$. Then 
\[
F(a)\ =\ \fm_F(a-I_F)\ \succeq\ \fm_F,
\]
so $vF(a) \leq vF$. For $u \in K$ with $|vu|<\beta$, we have $vF(u)= vF(a) \leq vF$, as desired. Now assume $r > 0$. Let $a \in K$ with $|va|<\beta$ and set
\[
\alpha\ \coloneqq \ r\inv\big(vF(\jet^ra) -vF\big).
\]
For $u \in K$ with $|vu|<\beta$, we have $vF(\jet^ru) =vF(\jet^ra) = vF+r\alpha$, so it suffices to show that $\alpha \in \Gamma(\der)$. Let $\epsilon\in \Gamma^>$ with $r\epsilon<\beta$. Using Lemma~\ref{lem-approxbound}, take $b \in K^\times$ and $\gamma \in \Gamma(\der)$ with $|vb|<r\epsilon<\beta$ and
\[
vF(\jet^rb)\ \leq \ vF+r\gamma+r\epsilon.
\]
By assumption, $vF(\jet^ra) = vF(\jet^rb)$, so 
\[
\alpha- \epsilon \ =\ r\inv\big(vF(\jet^ra) -vF\big)-\epsilon \ =\ r\inv\big(vF(\jet^rb) -vF\big)-\epsilon \ \leq\ \gamma\ \in\ \Gamma(\der).
\]
As $\epsilon$ can be taken to be arbitrarily small, we have $\alpha \in \Gamma(\der)$ by Lemma~\ref{lem-supisactive}.
\end{proof}

\subsection{Eventual smallness} \label{subsec-evensmall}
In this subsection, let $\phi$ range over $K^\times$, let $\ell\prec1$ be an element in a $\TdO$-extension of $K$, and suppose $v(\ell- K)$ has no largest element. We say that \textbf{$F$ is small near $(K,\ell)$} if $I_F(\jet^{r-1}y) \prec 1$ for all $y \in K$ sufficiently close to $\ell$.

\begin{lemma}
\label{lem-smallsetdownward}
Let $\phi_0\in K^\times$ with $v\phi_0 \in \Gamma(\der)$ and suppose $v\phi\leq v\phi_0$. If $F^{\phi_0}$ is small near $(K^{\phi_0},\ell)$, then $F^\phi$ is small near $(K^\phi,\ell)$.
\end{lemma}
\begin{proof}
By replacing $K$, $F$, and $\phi$ with $K^{\phi_0}$, $F^{\phi_0}$, and $\phi_0\inv \phi$, we may assume $\phi_0 = 1$ (so $K$ has small derivation) and $\phi\succeq 1$. Set $\derdelta \coloneqq \phi\inv\der$. Suppose $F$ is small near $(K,\ell)$ and let $y \in \smallo$ be close enough to $\ell$ so that $I_F(\jet^{r-1}y) \prec 1$. We claim that $I_{F^\phi}(\deltajet^{r-1}y)\prec 1$. By Lemma~\ref{lem-compconj}, we have
\[
I_{F^\phi}(\deltajet^{r-1}y)\ =\ \phi^{-r}\Big(I_F^\phi(\deltajet^{r-1}y)- \sum_{i=0}^{r-1}\xi_i^r(\phi)\derdelta^iy\Big)\ =\ \phi^{-r} I_F(\jet^{r-1}y)- \sum_{i=0}^{r-1}\phi^{-r}\xi_i^r(\phi)\derdelta^iy.
\]
As $I_F(\jet^{r-1}y) \prec 1$ and $\phi \succeq 1$, we have $\phi^{-r} I_F(\jet^{r-1}y)\prec 1$, so it remains to show
\[
\sum_{i=0}^{r-1}\phi^{-r}\xi_i^r(\phi)\derdelta^iy\ \prec \ 1.
\]
We claim that $\phi'/\phi \preceq \phi$. This is clear in the case that $\phi'\preceq \phi$, for then $\phi'/\phi \preceq 1 \preceq \phi$ (note that if $\phi \asymp 1$, then we are in this case by Fact~\ref{fact-bigtobig}). On the other hand, if $\phi' \succ \phi\succ 1$, then $\phi'/\phi \preceq \phi$ by~\cite[Lemma 6.4.1]{ADH17}. Now~\cite[Lemma 2.2]{ADH18B} gives $\phi^{-r}\xi_i^r(\phi) \preceq 1$ for each $i < r$. Since $K^\phi$ has small derivation and $y \prec 1$, we have $\phi^{-r}\xi_i^r(\phi)\derdelta^iy\prec 1$ for each $i<r$ as desired.
\end{proof}

\noindent
We say that \textbf{$F$ is eventually small near $(K,\ell)$} if $F^\phi$ is small near $(K^\phi,\ell)$ whenever $v\phi \in \Gamma(\der)$. Unlike smallness, eventual smallness is invariant under compositional conjugation: $F$ is eventually small near $(K,\ell)$ if and only if $F^\phi$ is eventually small near $(K^\phi,\ell)$. By the above lemma, the set
\[
\big\{v\phi\in \Gamma(\der):F^\phi\text{ is small near }(K^\phi,\ell)\big\}
\]
is a downward closed subset of $\Gamma(\der)$. Thus, $F$ is eventually small near $(K,\ell)$ if and only if $F^\phi$ is small near $(K^\phi,\ell)$ for all sufficiently large $v\phi \in \Gamma(\der)$, and $F$ is not eventually small near $(K,\ell)$ if and only if $F^\phi$ is not small near $(K^\phi,\ell)$ for all sufficiently large $v\phi \in \Gamma(\der)$. Eventual smallness serves as a crude analog of the Newton degree in~\cite{ADH17} and~\cite{ADH18B}; one should think of $F$ being eventually small as analogous to $F$ having positive Newton degree. Of course, Newton degree makes no sense for arbitrary definable functions.

\begin{lemma}
\label{lem-differentvals}
Suppose that $S(\der) = \{0\}$ and that $F$ is eventually small near $(K,\ell)$. For each $\beta \in \Gamma^>$, we have
\[
F(\jet^r a) \ \not\asymp\ F(\jet^rb)
\]
for some $a,b\in K$ with $va,vb >-\beta$.
\end{lemma}
\begin{proof}
Let $\beta \in \Gamma^>$ and let $a \in K$ with $|va|<\beta$. If $F(\jet^r a) \not\asymp F(\jet^rb)$ for some $b \in K$ with $|vb|<\beta$, then we are done, so we may assume $F(\jet^r a) \asymp F(\jet^rb)$ for all $b \in K$ with $|vb|<\beta$. Then Corollary~\ref{cor-sameval} gives $\gamma \in \Gamma(\der)$ with $vF(\jet^ra)\leq vF+r\gamma$. Take $\phi$ with $v\phi = \gamma$. Then $F^\phi$ is small near $(K^\phi,\ell)$, so we may take $y \in \smallo$ close enough to $\ell$ so that $I_{F^\phi}(\deltajet^{r-1} y) \prec 1$. Since $\derdelta^ry\prec 1$ as well, we have
\[
F(\jet^ry)\ =\ F^\phi(\deltajet^ry)\ =\ \phi^r\fm_F\big(\derdelta^ry-I_{F^\phi}(\deltajet^{r-1}y)\big) \ \prec\ \phi^r\fm_F,
\]
so $vF(\jet^ry) > vF+r\gamma\geq vF(\jet^ra)$.
\end{proof}

\section{Vanishing}\label{sec-vanishing}
\noindent
In this section, let $\phi$ range over $K^\times$, let $\ell$ be an element in a strict $\TdO$-extension $L$ of $K$, and suppose $v(\ell- K)$ has no largest element. Unlike in Subsection~\ref{subsec-evensmall}, we do not assume $\ell \prec 1$.

\begin{definition}
Let $F\colon K^{1+r} \to K$ be an $\cL(K)$-definable function in implicit form.
We say \textbf{$F$ vanishes at $(K,\ell)$} if $F_{+a,\times d}$ is eventually small near $\big(K,d\inv(\ell-a)\big)$ for all $a \in K$ and $d \in K^\times$ with $\ell-a \prec d$.
\end{definition}

\noindent
Let $Z(K,\ell)$ be the set of all $\cL(K)$-definable functions in implicit form which vanish at $(K,\ell)$. For each $r$, we let $Z_r(K,\ell)$ be the functions in $Z(K,\ell)$ of arity $1+r$, so $Z(K,\ell)$ is equal to the disjoint union $\bigcup_r Z_r(K,\ell)$. The set $Z(K,\ell)$ serves a similar purpose to the set in~\cite{ADH17} and~\cite{ADH18B} with the same name. Note that $Z(K,\ell)$ does not depend on $\ell$, only on the $\LO$-type of $\ell$ over $K$. We will show in Proposition~\ref{prop-imsim} below that if $Z(K,\ell) = \emptyset$ then $F(\jet^r\ell)\neq 0$ for all $\cL(K)$-definable functions $F$ in implicit form.

\begin{lemma}
\label{lem-zeroempty}
$Z_0(K,\ell) = \emptyset$.
\end{lemma}
\begin{proof}
Let $F\colon K\to K$ be an $\cL(K)$-definable function in implicit form, so $I_F \in K$. Let $d \in K^\times$ with $d\asymp \ell-I_F$ and let $a \in K$ with $\ell-a\prec d$. Then $I_F - a \asymp d$, so 
\[
I_{F_{+a,\times d}} \ =\ d\inv(I_F-a)\ \asymp\ 1
\]
and $F_{+a,\times d}$ is not small near $\big(K,d\inv(\ell-a)\big)$. As $F^\phi_{+a,\times d} = F_{+a,\times d}$ for all $\phi$, we see that $F_{+a,\times d}$ is not eventually small near $\big(K,d\inv(\ell-a)\big)$, so $F \not\in Z(K,\ell)$.
\end{proof}

\begin{lemma}
\label{lem-differentclosevals}
Suppose $S(\der) = \{0\}$, let $F \in Z_r(K,\ell)$, and let $y \in K$. Then there is $z \in K$ with $\ell-z\prec \ell-y$ and $F(\jet^ry)\not\asymp F(\jet^rz)$.
\end{lemma}
\begin{proof}
Let $a \in K$ and $d \in K^\times$ with $\ell-a \prec d\prec \ell-y$. Set $\beta \coloneqq vd- v(\ell-y)> 0$. Since $F_{+a,\times d}$ is eventually small near $\big(K,d\inv(\ell-a)\big)$, Lemma~\ref{lem-differentvals} gives $b_1,b_2 \in K$ with $vb_1,vb_2>-\beta$ and 
\[
vF_{+a,\times d}(\jet^rb_1)\ \neq\ vF_{+a,\times d}(\jet^rb_2).
\]
Either $vF_{+a,\times d}(\jet^rb_1)$ or $v F_{+a,\times d}(\jet^rb_2)$ is different from $v F(\jet^ry)$, so suppose $vF_{+a,\times d}(\jet^rb_1)\neq v F(\jet^ry)$ and set $z \coloneqq db_1+a$. Then $F(\jet^rz) = F_{+a,\times d}(\jet^rb_1)\not\asymp F(\jet^ry)$ and, since $v(db_1)> vd-\beta = v(\ell-y)$, we have 
\[
v(\ell - z)\ =\ v\big((\ell-a)- db_1\big)\ \geq\ \min\big\{v(\ell-a),v(db_1)\big\} \ >\ v(\ell-y).\qedhere
\]
\end{proof}

\subsection{Behavior of nonvanishing functions}
Fix $r$ and suppose $Z_q(K,\ell) = \emptyset$ for all $q<r$. Let $M$ be an arbitrary strict $\TdO$-extension of $K$. Our goal is to prove the following result:

\begin{proposition}
\label{prop-imsim}
Let $F\colon K^{1+r} \to K$ be an $\cL(K)$-definable function in implicit form with $F\not\in Z_r(K,\ell)$. Then $F(\jet^r\ell)\neq 0$ and
\[
F(\jet^r\ell)\ \sim \ F(\jet^ry)
\]
for all $y \in M$ sufficiently close to $\ell$.
\end{proposition}

\noindent
This proposition requires a somewhat lengthy proof by induction, so we make the following hypothesis.

\begin{induction*}[IH]We assume that for all $q<r$ and all $\cL(K)$-definable functions $F\colon K^{1+q} \to K$ in implicit form, we have $F(\jet^q\ell)\neq 0$ and
\[
F(\jet^q\ell)\ \sim \ F(\jet^qy)
\]
for all $y \in M$ sufficiently close to $\ell$.
\end{induction*}

\begin{lemma}\label{lem-isimm}
Suppose (IH) holds. Then $K\langle \jet^{r-1}\ell\rangle$ is an immediate $\TO$-extension of $K$. 
\end{lemma}
\begin{proof}
Let $q< r$ be given and assume $K\langle \jet^{q-1}\ell\rangle$ is an immediate $\TO$-extension of $K$ (this holds vacuously when $q = 0$). We will show that $K\langle \jet^q \ell\rangle$ is an immediate $\TO$-extension of $K\langle \jet^{q-1}\ell\rangle$, from which the lemma follows by induction. Let $G\colon K^q\to K$ be an $\cL(K)$-definable function. For all $u \in K$ sufficiently close to $\ell$, we have
\[
\ell^{(q)}- G(\jet^{q-1}\ell) \ \sim\ u^{(q)}- G(\jet^{q-1}u)\ \in \ K
\]
by (IH). Since $G$ is arbitrary, we may apply Lemma~\ref{lem-immediateext2} with $K\langle \jet^{q-1}\ell\rangle$ and $\ell^{(q)}$ in place of $K$ and $\ell$ to get that $K\langle \jet^q\ell\rangle$ is an immediate $\TO$-extension of $K\langle \jet^{q-1}\ell\rangle$.
\end{proof}

\begin{lemma}
\label{lem-incell}
Suppose (IH) holds, let $A \subseteq K^r$ be $\LO(K)$-definable, and suppose $\jet^{r-1}(\ell)\in A^L$. Then $A$ has nonempty interior and $\jet^{r-1}(y) \in A^M$ for all $y \in M$ sufficiently close to $\ell$.
\end{lemma}
\begin{proof}
Using Lemmas~\ref{lem-Lapprox} and~\ref{lem-isimm}, we take an $\cL(K)$-definable cell $D$ contained in $A$ with $\jet^{r-1}(\ell)\in D^L$. Let $q<r$ be given and assume $\pi_q(D)$ is open and $\jet^{q-1}(y) \in \pi_q(D^M)$ for all $y \in M$ sufficiently close to $\ell$ (this holds vacuously when $q = 0$). We will show that $\pi_{q+1}(D)$ is open and $\jet^q(y) \in \pi_{q+1}(D^M)$ for all $y \in M$ sufficiently close to $\ell$, from which the lemma follows by induction. If $\pi_{q+1}(D)$ is not open, then $\pi_{q+1}(D) = \Graph\big(G|_{\pi_q(D)}\big)$ for some $\cL(K)$-definable function $G\colon K^q \to K$, so $\ell^{(q)} = G(\jet^{q-1} \ell)$, contradicting (IH). Therefore, $\pi_{q+1}(D)$ is open. Suppose $\pi_{q+1}(D)$ is of the form 
\[
\big\{(a,b) :a \in \pi_q(D)\text{ and }G(a)<b<H(a)\big\}
\]
for some $\cL(K)$-definable functions $G,H\colon K^q \to K$. Then (IH) gives
\[
 y^{(q)}-G(\jet^{q-1}y) \ \sim \ \ell^{(q)}-G(\jet^{q-1}\ell)\ >\ 0\ >\ \ell^{(q)}-H(\jet^{q-1}\ell)\ \sim\ y^{(q)}-H(\jet^{q-1}y)
\]
for all $y \in M$ sufficiently close to $\ell$, so $G(\jet^{q-1}y)<y^{(q)}<H(\jet^{q-1}y)$ for these $y$. This gives $\jet^q(y) \in \pi_{q+1}(D^M)$ for all $y \in M$ sufficiently close to $\ell$ as desired. If $\pi_{q+1}(D)$ is of the form
\[
\big\{(a,b) :a \in \pi_q(D)\text{ and } b>G(a)\big\}\ \text{ or }\ \big\{(a,b) :a \in \pi_q(D)\text{ and } b<H(a)\big\},
\]
then a simpler version of the above argument works. If $\pi_{q+1}(D) = \pi_q(D)\times K$, then the result follows immediately from the inductive assumption.
\end{proof}

\begin{corollary}
\label{cor-generalclose}
Suppose (IH) holds and let $G\colon K^r \to K$ be an $\cL(K)$-definable function. If $G(\jet^{r-1}\ell)= 0$, then $G(\jet^{r-1}y)= 0$ for all $y \in M$ sufficiently close to $\ell$. If $G(\jet^{r-1}\ell)\neq 0$, then 
\[
G(\jet^{r-1}\ell)\ \sim \ G(\jet^{r-1}y)
\]
for all $y \in K$ sufficiently close to $\ell$. 
\end{corollary}
\begin{proof}
If $G(\jet^{r-1}\ell) = 0$, then apply Lemma~\ref{lem-incell} to the $\cL(K)$-definable set
\[
\big\{a \in K^r: G(a)=0\big\}.
\]
If $G(\jet^{r-1}\ell) \neq 0$, then since $K\langle \jet^{r-1}\ell\rangle$ is an immediate $\TO$-extension of $K$ by Lemma~\ref{lem-isimm}, we may take $g \in K^\times$ with $G(\jet^{r-1}\ell)\sim g$. Now apply Lemma~\ref{lem-incell} to the $\LO(K)$-definable set
\[
\big\{a \in K^r: G(a) \sim g\big\}.\qedhere
\]
\end{proof}

\noindent
We are now ready to prove Proposition~\ref{prop-imsim}.

\begin{proof}[Proof of Proposition~\ref{prop-imsim}]
Suppose Proposition~\ref{prop-imsim} holds with $q$ in place of $r$ for all $q<r$ (this is vacuous if $r = 0$). Then (IH) holds, since we are assuming that $Z_q(K,\ell) = \emptyset$ for all $q<r$. Let $F\colon K^{1+r} \to K$ be as in the statement of the proposition. Since $F \not\in Z_r(K,\ell)$, we may take $a \in K$ and $d \in K^\times$ with $\ell-a\prec d$ such that $F_{+a,\times d}$ is not eventually small near $\big(K,d\inv(\ell-a)\big)$. Set $e \coloneqq d\inv(\ell-a)\prec 1$ and take $\phi$ with $v\phi \in \Gamma(\der)$ such that $F_{+a,\times d}^\phi$ is not small near $(K^\phi,e)$. Set $\derdelta \coloneqq \phi\inv\der$ and set
\[
\fm\ \coloneqq \ \fm_{F_{+a,\times d}^\phi},\qquad G\ \coloneqq \ \big(I_{F_{+a,\times d}^\phi}\big)_{-d\inv a,\times d\inv}^{\phi\inv}.
\]
We have
\[
F(\jet^r\ell)\ =\ F_{+a,\times d}^\phi(\deltajet^re)\ =\ \fm_{F_{+a,\times d}^\phi}\big(\derdelta^re- I_{F_{+a,\times d}^\phi}(\deltajet^{r-1}e)\big)\ =\ \fm\big(\derdelta^re- G(\jet^{r-1}\ell)\big).
\]
Using Corollary~\ref{cor-generalclose}, we take $\eta \in v(\ell-K)$ such that for $y \in M$, if $y$ is $\eta$-close to $\ell$, then either
\[
G(\jet^{r-1}y)\ =\ G(\jet^{r-1}\ell)\ =\ 0\ \text{ or }\ G(\jet^{r-1}y)\ \sim\ G(\jet^{r-1}\ell)\ \neq\ 0.
\]
Then $\eta- vd \in v(e-K)$ and, since $e \prec 1$, we may increase $\eta$ and arrange $\eta-vd>0$. Since $F_{+a,\times d}^\phi$ is not small near $(K^\phi,e)$, we may take $z_0 \in K$ with
 \[
v(e-z_0)\ >\ \eta-vd,\qquad I_{F_{+a,\times d}^\phi}(\deltajet^{r-1}z_0)\ \succeq\ 1.
\]
Then $v\big(\ell-(dz_0+a)\big)>\eta$, so we have
\[
G(\jet^{r-1}\ell)\ \sim \ G\big(\jet^{r-1}(dz_0+a)\big)\ =\ I_{F_{+a,\times d}^\phi}(\deltajet^{r-1}z_0)\ \succeq\ 1.
\]
Since $\derdelta$ is small, $L$ strictly extends $K$, and $e \prec 1$, we have $\derdelta^re \prec 1$, so 
\[
F(\jet^r\ell)\ =\ \fm\big(\derdelta^re- G(\jet^{r-1}\ell)\big) \ \sim \ -\fm G(\jet^{r-1}\ell) \ \neq\ 0.
\]
Now, let $y \in M$ and suppose that $y$ is $\eta$-close to $K$. Set $z \coloneqq d\inv(y-a)$, so
\[
F(\jet^ry)\ =\ F_{+a,\times d}^\phi(\deltajet^rz)\ =\ \fm_{F_{+a,\times d}^\phi}\big(\derdelta^rz- I_{F_{+a,\times d}^\phi}(\deltajet^{r-1}z)\big)\ =\ \fm\big(\derdelta^rz- G(\jet^{r-1}y)\big).
\]
Then $z$ is $(\eta-vd)$-close to $e\prec 1$, so $z \prec 1$ and $\derdelta^rz \prec 1$. Since $y$ is $\eta$-close to $\ell$, we also have $G(\jet^{r-1}y) \sim G(\jet^{r-1}\ell) \succeq 1$. Thus,
\[
F(\jet^ry)\ =\ \fm\big(\derdelta^rz- G(\jet^{r-1}y)\big)\ \sim -\fm G(\jet^{r-1}y)\ \sim\ -\fm G(\jet^{r-1}\ell)\ \sim\ F(\jet^r\ell).\qedhere
\]
\end{proof}

\section{Constructing immediate extensions when $S(\der) = \{0\}$}\label{sec-trivial}
\noindent
As in the previous section, let $\ell$ be an element in a strict $\TdO$-extension $L$ of $K$ and suppose $v(\ell- K)$ has no largest element.

\begin{proposition}
\label{prop-zerosetisempty}
Suppose $Z(K,\ell) = \emptyset$. Then $K\langle \jet^\infty\ell \rangle$ is an immediate strict $\TdO$-extension of $K$. Let $b$ be an element in a strict $\TdO$-extension $M$ of $K$ with $v(b-y) = v(\ell-y)$ for each $y \in K$. Then there is a unique $\LdO(K)$-embedding $K\langle \jet^\infty\ell\rangle\to M$ sending $\ell$ to $b$.
\end{proposition}
\begin{proof}
By Lemma~\ref{lem-isimm}, $K\langle \jet^\infty\ell \rangle$ is an increasing union of immediate $\TO$-extensions of $K$, so it is itself an immediate $\TO$-extension of $K$. It is also strict, as $L$ is strict. As for the existence of an $\LdO(K)$-embedding $K\langle \jet^\infty\ell\rangle\to M$, we proceed by induction. Let $r\geq 0$ and suppose we have an $\LO(K)$-embedding $\imath\colon K\langle \jet^{r-1}\ell\rangle\to M$ which sends the tuple $\jet^{r-1}(\ell)$ to $ \jet^{r-1}(b)$ (this holds vacuously when $r = 0$). Let $G\colon K^r\to K$ be an $\cL(K)$-definable function. Since $b$ is $\eta$-close to $\ell$ for all $\eta \in v(\ell-K)$, Proposition~\ref{prop-imsim} gives
\[
\ell^{(r)}- G(\jet^{r-1}\ell)\ \sim \ b^{(r)}- G(\jet^{r-1}b).
\]
As $G$ is arbitrary and $\imath\big(G(\jet^{r-1}\ell)\big) = G(\jet^{r-1}b)$, 
we may apply Corollary~\ref{cor-immediateext1} with $K\langle \jet^{r-1}\ell\rangle$, $\ell^{(r)}$, and $b^{(r)}$ in place of $K$, $\ell$, and $b$ to extend $\imath$ to an $\LO(K)$-embedding $K\langle \jet^r\ell\rangle\to M$ sending $\ell^{(r)}$ to $ b^{(r)}$. The union of these embeddings is an $\LO(K)$-embedding $K\langle \jet^\infty\ell\rangle\to M$ which sends $\jet^\infty(\ell)$ to $\jet^\infty(b)$. This is even an $\LdO(K)$-embedding by Fact~\ref{fact-transext}. As an $\LdO(K)$-embedding, it is uniquely determined by the condition that $\ell$ be sent to $b$.
\end{proof}

\begin{proposition}
\label{prop-zerosetisnonempty}
Suppose $S(\der) = \{0\}$, let $F \in Z_{r+1}(K,\ell)$, and suppose that $Z_q(K,\ell) = \emptyset$ for all $q \leq r$. Then $K$ has an immediate strict $\TdO$-extension $K\langle \jet^ra\rangle$ with $F(\jet^{r+1}a) = 0$ and $v(a-y) = v(\ell-y)$ for each $y \in K$. Let $b$ be an element in a strict $\TdO$-extension $M$ of $K$ with $F(\jet^{r+1}b) = 0$ and $v(b-y) = v(\ell-y)$ for each $y \in K$. Then there is a unique $\LdO(K)$-embedding $K\langle \jet^ra\rangle\to M$ sending $a$ to $b$.
\end{proposition}
\begin{proof}
Let $(a_0,\ldots,a_r)$ realize the $\LO(K)$-type of $\jet^r(\ell)$ in some $\TO$-extension of $K$. Then $K\langle a_0,\dots,a_r\rangle$ is an immediate $\TO$-extension of $K$ by Lemma~\ref{lem-isimm}. The tuple $\jet^r(\ell)$ is $\cL(K)$-independent by Proposition~\ref{prop-imsim}, so $(a_0,\ldots,a_r)$ is $\cL(K)$-independent as well. Using Fact~\ref{fact-transext}, we extend $\der$ to a $T$-derivation on $K\langle a_0,\dots,a_r\rangle$ with $a_i' = a_{i+1}$ for $i <r$ and $a_r' = I_F(a_0,\ldots,a_r)$. Set $a \coloneqq a_0$ so $K\langle a_0,\dots,a_r\rangle= K \langle \jet^ra\rangle$ and $a^{(r+1)} = I_F(\jet^r a)$. 

We need to show that $K \langle \jet^ra\rangle$ is a strict extension of $K$. Let $\phi \in K^\times$ with $v\phi \in \Gamma(\der)$ and let $G\colon K^{r+1} \to K$ be an $\cL(K)$-definable function with $G(\jet^ra) \prec 1$. By Fact~\ref{fact-onlychecksmall}, it suffices to show that $G(\jet^ra)' \prec \phi$. We assume that $G(\jet^r a) \neq 0$, and we take an $\cL(K)$-definable open set $U \subseteq K^{1+r}$ on which $G$ is $\cC^1$ and which contains $\jet^r(a)$ in its natural extension. By Fact~\ref{fact-basicTderivation2}, we have
\[
G(\jet^r a)' \ =\ G^{[\der]}(\jet^r a)+\nabla G(\jet^r a)\cdot \big(a',a'',\ldots,a^{(r)},I_F(\jet^r a)\big).
\]
Let $Y = (Y_0,\ldots,Y_r)$ and let $H\colon U\to K$ be the function
\[
H(Y)\ \coloneqq \ G^{[\der]}(Y)+\nabla G(Y)\cdot \big(Y_1,\ldots,Y_r,I_F(Y)\big),
\]
so $H(\jet^ra) = G(\jet^ra)'$. Suppose toward contradiction that $H(\jet^ra) \succeq \phi$. Since $\jet^r(\ell)$ and $\jet^r(a)$ have the same $\LO(K)$-type and $Z_r(K,\ell) = \emptyset$, Lemma~\ref{lem-incell} and Corollary~\ref{cor-generalclose} give $\eta \in v(\ell-K)$ with
\[
\jet^r(y) \ \in\ U,\qquad G(\jet^ry)\ \sim\ G(\jet^r a) \ \prec\ 1,\qquad H(\jet^ry)\ \sim\ H(\jet^ra)\ \succeq\ \phi
\]
for all $y \in K$ which are $\eta$-close to $\ell$. For the remainder of this proof, we let $y\in K$ be $\eta$-close to $\ell$. Since $G(\jet^ry) \prec 1$ and $v\phi \in \Gamma(\der)$, we have
\[
G(\jet^ry)'\ =\ G^{[\der]}(\jet^ry)+\nabla G(\jet^ry)\cdot (y',\ldots,y^{(r)},y^{(r+1)})\ \prec\ \phi\ \preceq\ H(\jet^ry).
\]
Thus
\[
G(\jet^ry)'-H(\jet^ry)\ =\ \frac{\partial G}{\partial Y_r}(\jet^ry)\big(y^{(r+1)}-I_F(\jet^ry)\big)\ \sim\ -H(\jet^ry).
\]
Since $H(\jet^ry) \neq 0$, we have $\frac{\partial G}{\partial Y_r}(\jet^ry)\neq 0$, so 
\[
y^{(r+1)}-I_F(\jet^ry)\ \sim\ -H(\jet^ry)\left(\frac{\partial G}{\partial Y_r}(\jet^ry)\right)\inv.
\]
We have $H(\jet^ry) \sim H(\jet^ra)$ and, by increasing $\eta$, we may assume $\frac{\partial G}{\partial Y_r}(\jet^ry) \sim \frac{\partial G}{\partial Y_r}(\jet^ra)$. Thus
\[
F(\jet^{r+1}y) \ =\ \fm_F\big(y^{(r+1)}-I_F(\jet^ry)\big) \sim\ -\fm_FH(\jet^ra)\left(\frac{\partial G}{\partial Y_r}(\jet^ra)\right)\inv. 
\]
In particular, $F(\jet^{r+1}y) \sim F(\jet^{r+1}z)$ for all $y,z \in K$ with $v(\ell-y),v(\ell-z)> \eta$, contradicting Lemma~\ref{lem-differentclosevals}.

Now let $M$ and $b$ be as in the statement of the proposition. Since $\jet^r(\ell)$ and $\jet^r(a)$ have the same $\LO(K)$-type and $Z_r(K,\ell) = \emptyset$, we may construct an $\LO(K)$-embedding $K\langle \jet^ra\rangle \to M$ which sends $\jet^r(a)$ to $ \jet^r(b)$ as in the proof of the previous proposition. This is even an $\LdO(K)$-embedding by Fact~\ref{fact-transext}. As an $\LdO(K)$-embedding, it is uniquely determined by the condition that $a$ be sent to $b$.
\end{proof}

\begin{theorem}
\label{thm-main1}
Suppose $S(\der) = \{0\}$. Then $K$ has a spherically complete immediate strict $\TdO$-extension.
\end{theorem}
\begin{proof}
We may assume that $K$ is not itself spherically complete. It suffices to show that $K$ has a proper immediate strict $\TdO$-extension, as the property $S(\der) = \{0\}$ is preserved by immediate strict extensions. Let $\cB$ be a nested collection of closed $v$-balls in $K$ with empty intersection in $K$ and let $\ell$ be an element in an elementary $\TdO$-extension of $L$ of $K$ with $\ell \in \bigcap \cB^L$. Then $v(\ell-K)$ has no largest element by Lemma~\ref{lem-basicvball}. If $Z(K,\ell)=\emptyset$, then $K\langle \jet^\infty\ell \rangle$ is a proper immediate strict $\TdO$-extension of $K$ by Proposition~\ref{prop-zerosetisempty}. Suppose $Z(K,\ell)\neq\emptyset$. Lemma~\ref{lem-zeroempty} gives $Z_0(K,\ell) = \emptyset$, so take $r$ maximal such that $Z_q(K,\ell) = \emptyset$ for all $q\leq r$. Then Proposition~\ref{prop-zerosetisnonempty} provides a proper immediate strict $\TdO$-extension $K\langle\jet^r a\rangle$ of $K$ where $a$ is in the natural extension of each $B \in \cB$.
\end{proof}

\noindent
Before moving on, let us consider a couple of cases that can be handled by Theorem~\ref{thm-main1}. Suppose that $K$ has small derivation and that the induced derivation on $\res(K)$ is nontrivial. Then $\Gamma(\der) = \Gamma^{\leq}$, so $S(\der) = \{0\}$; see~\cite[Corollary 1.7 and Lemma 1.15]{ADH18B}. Given a $\TdO$-extension $M$ of $K$, it follows from Fact~\ref{fact-bigtobig} that $M$ is a strict extension of $K$ if and only if $M$ has small derivation. Thus, we have the following:

\begin{corollary}
\label{cor-nontrivres}
If $K$ has small derivation and the induced derivation on $\res(K)$ is nontrivial, then $K$ has a spherically complete immediate $\TdO$-extension with small derivation.
\end{corollary}

\noindent
Another important case where $S(\der) = \{0\}$ is when $K$ is asymptotic. Recall from~\cite{ADH17} that $K$ is said to be \textbf{asymptotic} if for all $f,g \in K^\times$ with $f,g\prec 1$, we have $f \prec g \Longleftrightarrow f' \prec g'$. If $K$ is asymptotic, then $S(\der) = \{0\}$ by~\cite[Lemma 1.14]{ADH18B}, so $K$ has a spherically complete immediate strict $\TdO$-extension. Moreover, if $K$ is asymptotic and $M$ is an immediate $\TdO$-extension of $K$, then $M$ is a strict $\TdO$-extension of $K$ if and only if $M$ is itself asymptotic by~\cite[Lemmas 1.11 and 1.12]{ADH18B}. We summarize below:

\begin{corollary}
\label{cor-asymp}
If $K$ is asymptotic, then $K$ has a spherically complete immediate asymptotic $\TdO$-extension.
\end{corollary}

\section{Coarsening by $S(\der)$}\label{sec-coarsening}
\noindent
In this section, we prove our main theorem. First, we establish some results on residue field extensions.

\begin{lemma}
\label{lem-smalltosmall}
Suppose $K$ has small derivation and let $L = K\langle a \rangle$ be a simple $\TdO$-extension of $K$ with $a \asymp 1$, $\bar{a} \not\in \res(K)$, and $a' \preceq 1$. Then $L$ has small derivation. Moreover, if $\der \cO \subseteq \smallo$ and $a' \prec 1$, then $\der_L \cO_L \subseteq \smallo_L$.
\end{lemma}
\begin{proof}
Let $F\colon K\to K$ be an $\cL(K)$-definable function with $F(a) \prec 1$. We need to show that
\[
F(a)' \ =\ F^{[\der]}(a)+ F'(a)a'\ \prec \ 1.
\]
Since $\res(L) \neq \res(K)$ and $T$ is power bounded, we have $\Gamma_L = \Gamma$ by~\cite[Corollary 5.6]{vdD97} (the Wilkie inequality for power bounded theories). Using Lemma~\ref{lem-smallderivres}, we see that $F'(a)\preceq F(a) \prec 1$, so $F'(a)a' \prec 1$ and it remains to show that $F^{[\der]}(a) \prec 1$. By $\LO$-elementarity, it suffices to show that for any $\LO(K)$-definable set $A \subseteq \cO$ with $a \in A^L$, there is $y \in A$ with $F^{[\der]}(y) \prec 1$. Let $A$ be such a set and, by shrinking $A$ if need be, assume that $F$ is $\cC^1$ on $A$ and that $F(y) \prec 1$ for all $y \in A$. Since $F'(a)\prec 1$, we can use $\LO$-elementarity to take $y \in A$ with $F'(y) \prec 1$. Since $y' \preceq 1$ by Fact~\ref{fact-bigtobig}, we have $F'(y)y' \prec 1$. Since $F(y)' \prec 1$ as well, this gives
\[
F^{[\der]}(y) \ =\ F(y)' - F'(y)y' \ \prec \ 1.
\]
This takes care of the first part of the lemma.

For the second part, assume that $\der \cO \subseteq \smallo$ and that $a' \prec 1$. We need to show that $F(a)' \prec 1$ for each $\cL(K)$-definable function $F\colon K\to K$ with $F(a) \preceq 1$. The proof is essentially the same as the proof of the first part, but now Lemma~\ref{lem-smallderivres} only gives that $F'(a)\preceq 1$. We make up for this by using our assumption that $\der \cO \subseteq \smallo$ and that $a' \prec 1$.
\end{proof}

\noindent
The following corollary serves as an analog of~\cite[Corollary 6.7]{ADH18B}.

\begin{corollary}
\label{cor-resext}
Suppose $\der\cO\subseteq \smallo$ and let $E$ be a $T$-extension of $\res(K)$. Then there is a strict $\TdO$-extension $L$ of $K$ such that $\Gamma_L= \Gamma$, the derivation on $\res(L)$ is trivial, and $\res(L)$ is $\cL(\res K)$-isomorphic to $E$.
\end{corollary}
\begin{proof}
It suffices to consider the case $E = \res(K)\langle f\rangle$ where $f \not\in \res(K)$. Let $L = K\langle a\rangle$ be a simple $T$-extension of $K$ where $a$ realizes the cut
\[
\{y\in K: y <\cO \text{ or } y \in \cO\text{ and } \bar{y}<f\}.
\]
We expand $L$ to an $\LO$-structure by letting
\[
\cO_L\ \coloneqq \ \big\{y\in L:|y|<d\text{ for all }d \in K\text{ with } d >\cO\big\}.
\]
This expansion of $L$ is a $\TO$-extension of $K$ by~\cite[Main Lemma 3.6]{DL95}. Note that $a \in \cO_L$ and that $\res(L) = \res(K)\langle \bar{a}\rangle$ is $\cL(\res K)$-isomorphic to $E$, since $\bar{a}$ and $f$ realize the same cut in $\res(K)$. In particular, $\res(L)\neq \res(K)$, so $\Gamma_L = \Gamma$ by the Wilkie inequality. Using Fact~\ref{fact-transext}, we extend $\der$ uniquely to a $T$-derivation on $L$ with $a' = 0$. We claim that $L$ is a strict $\TdO$-extension of $K$. Let $\phi \in K^\times$ and note that $\phi\inv a' = a' = 0$. If $\der\smallo \subseteq \phi\smallo$, then Lemma~\ref{lem-smalltosmall} applied to $K^\phi$ and $L^\phi$ in place of $K$ and $L$ gives $\der_L\smallo_L\subseteq \phi\smallo_L$. Likewise, if $\der\cO \subseteq \phi\smallo$, then Lemma~\ref{lem-smalltosmall} gives $\der_L\cO_L\subseteq \phi\smallo_L$. The case $\phi = 1$ gives that the derivation on $\res(L)$ is trivial.
\end{proof}

\noindent
We are now ready to prove the main theorem:

\begin{theorem}\label{thm-main2}
Suppose $S(\der)$ is a $\Lambda$-subspace of $K$. Then $K$ has an immediate strict $\TdO$-extension which is spherically complete.
\end{theorem}
\begin{proof}
If $S(\der)= \{0\}$, then we are done by Theorem~\ref{thm-main1}, so we may assume that $\Delta \coloneqq S(\der) \neq \{0\}$. We arrange by compositional conjugation that $K$ has small derivation. The assumption that $\Delta$ is a $\Lambda$-subspace of $K$ allows us to coarsen by $\Delta$, which we do, yielding the $\TdO$-model $K_\Delta$. The derivation on $\res(K_\Delta)$ is trivial by~\cite[Lemma 6.1]{ADH18B} and $S_{K_\Delta}(\der) = \{0\}$ by~\cite[Lemma 6.2]{ADH18B}. Let $E$ be a spherically complete immediate $\TO$-extension of $\res(K_\Delta)$; such an extension exists by Corollary~\ref{cor-immediateext}. Using Corollary~\ref{cor-resext}, we take a strict $\TdO$-extension $L$ of $K_\Delta$ such that $\Gamma_L =\dot{\Gamma}$, the derivation on $\res(L)$ is trivial, and $\res(L)$ is $\cL(\res K_\Delta)$-isomorphic to $E$. Then $S_L(\der) = \{0\}$ as well, and we apply Theorem~\ref{thm-main1} to $L$ to get a spherically complete immediate strict $\TdO$-extension $M$ of $L$. We have $\res(M) = \res(L)$ as $T$-models, so $\res(M)$ is $\cL(\res K_\Delta)$-isomorphic to $E$. We equip $\res(M)$ with a $T$-convex valuation ring $\cO_{\res(M)}$ so that $\res(M)$ is $\LO(\res K_\Delta)$-isomorphic to $E$; then $\res(M)$ is a spherically complete immediate $\TO$-extension of $\res(K_\Delta)$. Now let $M^*$ be the $\TdO$-model with underlying $\Td$-model $M$ and $T$-convex valuation ring
\[
\cO_{M^*}\ \coloneqq \ \{a \in \cO_M:\bar{a} \in \cO_{\res(M)}\}.
\]
Then $M^*$ is an immediate $\TdO$-extension of $K$ with $M^*_\Delta = M$; see Subsection~\ref{subsec-TOcoarsening}. By~\cite[Lemma 6.4]{ADH18B}, $M^*$ is a strict $\TdO$-extension of $K$. As $M^*_\Delta = M$ and $\res(M^*_\Delta) = \res(M)$ are both spherically complete, $M^*$ is spherically complete by Fact~\ref{fact-coarsespeccomp}. 
\end{proof}

\noindent
The diagram below catalogs the objects and maps involved in the proof of Theorem~\ref{thm-main2}. As in Subsection~\ref{subsec-TOcoarsening}, horizontal arrows are embeddings in the indicated language and downward arrows are places. Isomorphisms are labeled as such and every square commutes. 
\[
\begin{tikzcd}[row sep=large]
K_\Delta \arrow[r,"\text{strict}","\LdO" '] \arrow[d]& L\arrow[r,"\text{strict}","\LdO" '] \arrow[d]& M \arrow[d]&\hspace{-3cm}=M^*_\Delta \\
\res(K_\Delta)\arrow[r,"\cL" '] \arrow[d]\arrow[rr, bend right=20,"\LO" ']& \res(L)\arrow[r,"\sim","\cL" '] & \res(M) \arrow[r,"\sim","\LO" ']\arrow[d]& E \arrow[d]\\
\res(K) \arrow[rr,"\sim","\cL" '] &&\res(M^*)\arrow[r,"\sim","\cL" ']& \res(E) 
\end{tikzcd}
\]
Note that if $S(\der)$ is a $\Lambda$-subspace of $\Gamma$, then the same is true for any $\TdO$-model which is elementarily equivalent to $K$, as this is a property of the $\LdO$-theory of $K$. If $\Lambda$ is archimedean, then $S(\der)$ is always a $\Lambda$-subspace of $\Gamma$, so we have the following corollary:

\begin{corollary}
\label{cor-polycase}
If $T$ is polynomially bounded, then $K$ has an immediate strict $\TdO$-extension which is spherically complete.
\end{corollary}

\subsection{Uniqueness}
Let $M$ be a spherically complete immediate strict $\TdO$-extension of $K$. It is natural to ask: under which circumstances is $M$ the \emph{unique} such extension up to $\LdO(K)$-isomorphism? Uniqueness holds if $K$ is itself spherically complete, for then $M = K$. If $\der$ is trivial, then any immediate strict $\TdO$-extension of $K$ has trivial derivation as well, so $K$ has a unique spherically complete immediate strict $\TdO$-extension up to $\LdO(K)$-isomorphism by Corollary~\ref{cor-immediateext}.

\medskip\noindent
Suppose $T = \RCF$. Suppose also that $\der$ is small and that $\res(K)$ is linearly surjective (that is, for each $a_0,\ldots,a_r \in \res(K)$, there is $y \in \res(K)$ with $1+a_0y+\cdots+a_ry^{(r)} = 0$). If $K$ is monotone or if $K$ is asymptotic, then $M$ is unique up to isomorphism over $K$ by~\cite[Section 7.4]{ADH17} and~\cite{PC20}, respectively. Linear surjectivity seems to be essential here---in the final section of~\cite{ADH18B}, the authors construct a real closed $H$-field $R$ with small derivation which does not have a unique spherically complete immediate strict $\RCF^{\cO,\der}$-extension up to isomorphism over $R$. All $H$-fields are asymptotic, but no $H$-field has linearly surjective residue field (as the induced derivation on the residue field of an $H$-field is always trivial).

\medskip\noindent
When $T \neq \RCF$, nothing is known about uniqueness outside of the trivial cases. All the results in the case $T = \RCF$ depend crucially on \emph{differential henselianity}, a differential-algebraic property which we have yet to generalize to our setting.


\end{document}